\documentclass[12pt]{amsart}
\usepackage{amssymb}
\usepackage{hyperref}
\usepackage[all]{xy}

\newtheorem{theorem}{Theorem}[section]
\newtheorem{definition}[theorem]{Definition}
\newtheorem{proposition}[theorem]{Proposition}
\newtheorem{lemma}[theorem]{Lemma}

\begin{document}

\title{A duality principle for noncommutative cubes and spheres}

\author{Teodor Banica}
\address{T.B.: Department of Mathematics, Cergy-Pontoise University, 95000 Cergy-Pontoise, France. {\tt teodor.banica@u-cergy.fr}}


\begin{abstract}
We discuss a general duality principle, between noncommutative analogues of the standard cube $\mathbb Z_2^N$, and nonocommutative analogues of the standard sphere $S^{N-1}_\mathbb R$. This duality is by construction of algebraic geometric nature, and conjecturally connects the corresponding quantum isometry groups, taken in an affine sense.
\end{abstract}

\maketitle

\section*{Introduction}

Woronowicz axiomatized in \cite{wo1}, \cite{wo2} the compact quantum groups, and explained how the Schur-Weyl problem can be solved for the deformations of $SU_N$. This paved the way for a number of further developements.  Wang discovered in \cite{wa1}, \cite{wa2} the free quantum groups $O_N^+,U_N^+,S_N^+$, whose Tannakian duals were computed in \cite{ba1}, \cite{ba2}. Later on, a link was made with the work of Bichon \cite{bic} and Collins \cite{col}, and a systematic study, mainly focusing on the symmetry groups $G\subset S_N^+$ of finite graphs, was developed. See \cite{bbc}.

Goswami's axiomatization in \cite{go1}, \cite{go2} of the quantum isometry groups made it possible to reconcile the continuous computations in \cite{ba1}, for $O_N^+,U_N^+$, with the various computations for $S_N^+$ and its subgroups, as those in \cite{ba2}, \cite{bbc}. The point indeed is that $O_N^+,U_N^+$ appear as quantum isometry groups of the free spheres $S^{N-1}_{\mathbb R,+}$, $S^{N-1}_{\mathbb C,+}$, while $S_N^+$ and related quantum groups appear from discrete manifolds. See \cite{ba3}, \cite{ba4}, \cite{bgo}, \cite{bg1}, \cite{bgs}.

At the level of potential applications, the link with Connes' work \cite{co1}, \cite{co2} brought as well a substantial upgrade. Indeed, while the classical, connected manifolds cannot have genuine quantum isometries \cite{dgj}, for noncommutative manifolds like the Standard Model one \cite{cc1}, \cite{cc2} the quantum isometry group is bigger than the usual isometry group, containing therefore ``hidden'' symmetries, worth to be investigated. See \cite{bdd}, \cite{bd+}.

Short after the unification coming from \cite{go1}, the representation theory program for quantum isometry groups got once again ``dispersed'', this time due to a key connection with Voiculescu's free probability theory \cite{vdn}. K\"ostler and Speicher discovered in \cite{ksp} that a free de Finetti theorem holds, with $S_N$ replaced by $S_N^+$. Curran found a bit later a more advanced proof, and generalizations, using the Weingarten formula \cite{cu1}, \cite{cu2}. These results, along with \cite{spe}, suggested a whole new approach to probabilistic invariance questions, by axiomatizing and classifying the compact quantum groups having an ``elementary'' Tannakian dual, and then by studying the actions of such quantum groups on random variables. The axiomatization and some preliminary classification work were done in \cite{bc1}, \cite{bsp}, \cite{bve}, \cite{web}, and the corresponding invariance questions were investigated in \cite{bc2}. The whole idea ended up in producing a very active field of research. See \cite{bdu}, \cite{csp}, \cite{dko}, \cite{fre}, \cite{fwe}, \cite{rw1}, \cite{rw2}, \cite{rw3}, \cite{rw4}.

Regarding now the original geometric motivations, which are somehow obscured by the combinatorial axiomatization in \cite{bsp}, there have been several advances here:
\begin{enumerate}
\item The quantum isometries of various noncommutative spheres were investigated in \cite{ba3}, \cite{ba4}, \cite{bgo}, \cite{bg1}, \cite{bg2}. In all cases the quantum groups found are covered by the formalism in \cite{bsp}, or appear as deformations of such quantum groups.

\item The quantum isometries of various group duals were investigated in \cite{bs1}, \cite{bs2}, \cite{gma}, \cite{lso}, \cite{rw2}. Once again, for the basic examples, the quantum groups found are covered by the formalism in \cite{bsp}, and its 2-parametric extensions.
\end{enumerate}

The aim of the present paper is that of linking (1,2) by a general duality principle. The idea is very simple. Consider the standard cube $Y_N=\{-1,1\}^N\subset\mathbb R^N$. We have then an isomorphism $C^*(\mathbb Z_2^N)\simeq C(Y_N)$, given by $g_i\to x_i$, which gives an identification $\widehat{\mathbb Z_2^N}\simeq Y_N$. By rescaling by $1/\sqrt{N}$ we obtain an embedding $\widehat{\mathbb Z_2^N}\subset S^{N-1}_\mathbb R$, as follows:
$$\begin{matrix}
S^{N-1}_\mathbb R&=&\left\{x\in\mathbb R^N\Big|\sum_ix_i^2=1\right\}\\
\\
\bigcup&&\bigcup\\
\\
\widehat{\mathbb Z_2^N}&\simeq&\left\{x\in\mathbb R^N\Big|x_i=\pm\frac{1}{\sqrt{N}},\forall i\right\}
\end{matrix}$$

The point now is that this embedding appears as the $\Gamma=\mathbb Z_2^N$ particular case of a general inclusion of type $\widehat{\Gamma}\subset S_\Gamma$, where $\Gamma=<g_1,\ldots,g_N>$ is a reflection group, satisfying certain uniformity assumptions. Based on this remark, we will develop some theory:
\begin{enumerate}
\item First, we will extend the undeformed noncommutative sphere formalism in \cite{ba3}, \cite{ba4}, \cite{bgo}, as to cover objects of type $S_\Gamma$, as well as their twists $\bar{S}_\Gamma$.

\item We will study the spaces of type $\widehat{\Gamma}$, with $\Gamma=<g_1,\ldots,g_N>$ being a reflection group as above, that we call here ``noncommutative cubes''.

\item We will establish a correspondence $\widehat{\Gamma}\leftrightarrow S_\Gamma$, and we will discuss the comparison of the corresponding quantum isometry groups, taken in an affine sense.
\end{enumerate}

There is in fact a lot of work to be done here. We have as well a number of conjectural statements on the subject, for the most regarding the geometric realization, as quantum isometry groups, of the easy quantum groups $H_N\subset G\subset O_N^+$, and their twists.

We refer to the body of the paper for the precise statements of the results. The proofs are based on our previous work on noncommutative spheres in \cite{ba3}, \cite{ba4}, \cite{bgo}, \cite{bg1}, and on the classification work of Raum and Weber in \cite{rw1}, \cite{rw2}, \cite{rw3}, \cite{rw4}. Let us also mention that, at the axiomatic level, we use a formalism inspired from \cite{chi}, \cite{go2}, \cite{qsa}.

There are many questions raised by the present work. Here are some of them:
\begin{enumerate}
\item The spheres and other manifolds that we consider here are ``undeformed''. In the deformed case there are many interesting examples, see e.g. \cite{cdu}, \cite{cla}, \cite{ddl}, \cite{pod}. This raises the non-trivial question of ``deforming'' the present work.

\item Our manifolds are algebraic, and the study of their singularities/smoothness, and Riemannian aspects, remains an open problem. There are several questions here, in relation with \cite{cfk}, \cite{co1}, \cite{co2}, \cite{dgo}, already discussed in \cite{ba3}, \cite{ba4}.

\item We are dealing here with very simple manifolds, generalizing the unit cube and sphere. One interesting question regards the general geometric formulation of the notions of liberation and half-liberation, coming from \cite{bsp}, \cite{bpa}, \cite{bdd}, \cite{bdu}. 
\end{enumerate}

Further questions concern the unitary extension of the present work. Nor do we know on how to best interpret the probabilistic invariance questions studied in \cite{bc2} and in subsequent papers, as to make them fit into the present geometric setting.

The paper is organized as follows: in 1-2 we discuss the easy quantum groups $H_N\subset G\subset O_N^+$, in 3-4 we study the noncommutative cubes and spheres, and in 5-6 we present the duality principle, along with a number of consequences and extensions.

\medskip

\noindent {\bf Acknowledgements.} I would like to thank Steve Curran, Adam Skalski and Roland Speicher for various useful discussions, Alexandru Chirvasitu for a key remark regarding the quantum isometry groups, Sven Raum and Moritz Weber for keeping me informed on the advances in their classification work, and an anonymous referee for valuable suggestions. This work was partly supported by the NCN grant 2012/06/M/ST1/00169.

\section{Easy quantum groups}

We first recall the axiomatization of the easy quantum groups, from \cite{bsp}. We denote by $P(k,l)$ the set of partitions between an upper row of $k$ points, and a lower row of $l$ points. We will regard the elements of $P(k,l)$ in a pictorial way, with the upper and lower points, called ``legs'', connected by the blocks of the partition, called ``strings''.

The elements of $P(k,l)$ naturally act on tensors, as follows:

\begin{definition}
Associated to a partition $\pi\in P(k,l)$ is the linear map
$$T_\pi(e_{i_1}\otimes\ldots\otimes e_{i_k})=\sum_{j:\ker(^i_j)\leq\pi}e_{j_1}\otimes\ldots\otimes e_{j_l}$$
where $e_1,\ldots,e_N$ is the standard basis of $\mathbb C^N$.
\end{definition}

Here the kernel of a multi-index $(^i_j)=(^{i_1\ldots i_k}_{j_1\ldots j_l})$ is by definition the partition $\tau\in P(k,l)$ obtained by joining the sets of equal indices. Thus, the condition $\ker(^i_j)\leq\pi$ simply tells us that the strings of $\pi$ must join equal indices. Here are a few examples:
$$T_{|\,|}(e_i\otimes e_j)=e_i\otimes e_j\quad,\quad
T_{\slash\!\!\!\backslash}(e_i\otimes e_j)=e_j\otimes e_i$$
$$T_\cap(1)=\sum_ie_i\otimes e_i\quad,\quad
T_\cup(e_i\otimes e_j)=\delta_{ij}$$

Now let $O_N^+$ be the free analogue of $O_N$, constructed by Wang in \cite{wa1}. This is by definition the abstract spectrum of the universal algebra $C(O_N^+)$ generated by the entries of a $N\times N$ matrix $u=(u_{ij})$ which is orthogonal ($u=\bar{u},u^t=u^{-1}$), with comultiplication $\Delta(u_{ij})=\sum_ku_{ik}\otimes u_{kj}$, counit $\varepsilon(u_{ij})=\delta_{ij}$ and antipode $S(u_{ij})=u_{ji}$. We have:

\begin{definition}
A compact quantum group $G\subset O_N^+$ is called easy when 
$$Hom(u^{\otimes k},u^{\otimes l})=span(T_\pi|\pi\in D(k,l))$$
for any $k,l\in\mathbb N$, for certain subsets $D(k,l)\subset P(k,l)$.
\end{definition}

In other words, we call $G$ easy when its Schur-Weyl category, formed by the linear spaces $Hom(u^{\otimes k},u^{\otimes l})$, appears in the simplest possible way: from partitions.

The above subsets $D(k,l)\subset P(k,l)$ are not unique. In order to make them unique, we can ``saturate'', i.e. replace them by the biggest possible subsets $\widetilde{D}(k,l)\subset P(k,l)$ making the span formula hold. With this replacement made, $D=\bigcup_{kl}D(k,l)$ has a number of remarkable properties, and we say that we have a category of partitions. See \cite{bsp}.

We will be interested in the intermediate easy quantum groups $H_N\subset G\subset O_N^+$, where $H_N$ is the hyperoctahedral group. The main examples here are as follows:

\begin{proposition}
We have easy quantum groups $H_N\subset G\subset O_N^+$ as follows,
$$\xymatrix@R=15mm@C=15mm{
O_N\ar[r]&O_N^*\ar[r]&O_N^+\\
H_N\ar[r]\ar[u]&H_N^*\ar[r]\ar[u]&H_N^+\ar[u]}
\ \ \xymatrix@R=9mm@C=5mm{\\ \ar@{~}[r]&\\&\\}\ \ 
\xymatrix@R=15mm@C=15mm{
P_2\ar[d]&P_2^*\ar[l]\ar[d]&NC_2\ar[l]\ar[d]\\
P_{even}&P_{even}^*\ar[l]&NC_{even}\ar[l]}$$
with the diagram at right describing the corresponding categories of partitions.
\end{proposition}

We refer to \cite{bc1} for details. Let us just mention that $H_N^+$ is the quantum group constructed in \cite{bbc}, that $O_N^*\subset O_N^+$, $H_N^*\subset H_N^+$ are obtained by assuming that the standard coordinates $u_{ij}$ satisfy the half-commutation relations $abc=cba$, and that $P_{even}^*\subset P_{even},P_2^*\subset P_2$ consist of partitions having the property that when labelling counterclockwise the legs $\circ\bullet\circ\bullet\ldots$, each block has an equal number of black and white legs.

There are many other examples of easy quantum groups $H_N\subset G\subset O_N^+$, and we will need in what follows quite a substantial amount of information about such quantum groups, including their classification, coming from \cite{rw4}. Let us begin with:

\begin{definition}
We let $P_{even}^{[\infty]}$ be the category generated by the partition
$$
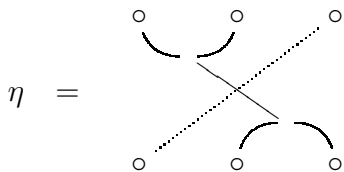
\xymatrix@R=2mm@C=3mm{\\ \\ \eta\ \ =\\ \\}\ \ \ 
\xymatrix@R=2mm@C=3mm{
\circ\ar@/_/@{-}[dr]&&\circ&&\circ\ar@{.}[ddddllll]\\
&\ar@/_/@{-}[ur]\ar@{-}[ddrr]\\
\\
&&&\ar@/^/@{-}[dr]\\
\circ&&\circ\ar@/^/@{-}[ur]&&\circ}$$
and we denote by $H_N^{[\infty]}$ the corresponding easy quantum group $H_N\subset G\subset H_N^+$.
\end{definition}

The elements $\pi\in P_{even}^{[\infty]}$ can be characterized by the fact that all their subpartitions $\sigma\subset\pi$ satisfy $\sigma\in P_{even}^*$. As an example, the verification of $\eta\in P_{even}^{[\infty]}$ goes as follows:
$$
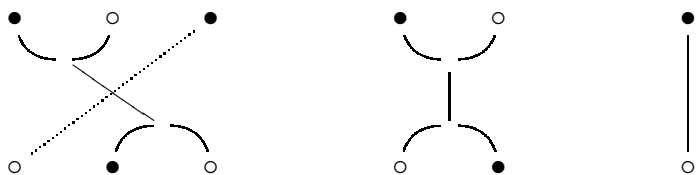
\xymatrix@R=2mm@C=3mm{
\bullet\ar@/_/@{-}[dr]&&\circ&&\bullet\ar@{.}[ddddllll]\\
&\ar@/_/@{-}[ur]\ar@{-}[ddrr]\\
\\
&&&\ar@/^/@{-}[dr]\\
\circ&&\bullet\ar@/^/@{-}[ur]&&\circ}\qquad\qquad\quad
\xymatrix@R=2mm@C=3mm{
\bullet\ar@/_/@{-}[dr]&&\circ\\
&\ar@/_/@{-}[ur]\ar@{-}[dd]\\
\\
&\ar@/^/@{-}[dr]\\
\circ\ar@/^/@{-}[ur]&&\bullet}\qquad\qquad\quad
\xymatrix@R=15.4mm@C=7mm{\bullet\ar@{-}[d]\\ \circ}
$$

Regarding now the quantum group $H_N^{[\infty]}$, it is known that this contains $H_N^*$, and also that $H_N^{[\infty]}\subset O_N^+$ appears by assuming that the standard coordinates $u_{ij}$ satisfy the relations $abc=0$, for any $a\neq c$ on the same row or column of $u$. See \cite{bc1}.

The point with $H_N^{[\infty]}$ comes from the following result:

\begin{proposition}
The easy quantum groups $H_N\subset G\subset O_N^+$ are as follows,
$$
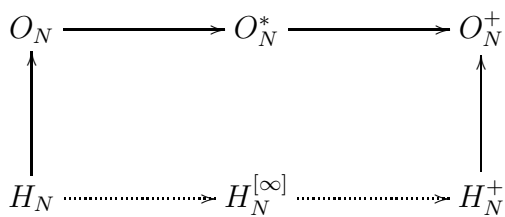
\xymatrix@R=15mm@C=20mm{
O_N\ar[r]&O_N^*\ar[r]&O_N^+\\
H_N\ar@.[r]\ar[u]&H_N^{[\infty]}\ar@.[r]&H_N^+\ar[u]}$$
with the dotted arrows indicating that we have intermediate quantum groups there.
\end{proposition}

This is a key result in the classification of easy quantum groups:

(1) The first dichotomy, $O_N\subset G\subset O_N^+$ vs. $H_N\subset G\subset H_N^+$, comes from the early classification results, from \cite{bc1}, \cite{bsp}, \cite{bve}, \cite{web}. In addition, these results solve as well the first problem, $O_N\subset G\subset O_N^+$, with $G=O_N^*$ being the unique non-trivial solution.

(2) The second dictotomy, $H_N\subset G\subset H_N^{[\infty]}$ vs. $H_N^{[\infty]}\subset G\subset H_N^+$, comes from \cite{rw1}, \cite{rw2}, \cite{rw3}, \cite{rw4}, and more specifically from the final classification paper \cite{rw4}, where the quantum groups $S_N\subset G\subset H_N^+$ with $G\not\subset H_N^{[\infty]}$ were classified, and shown to contain $H_N^{[\infty]}$.

Regarding now the case $H_N^{[\infty]}\subset G\subset H_N^+$, the precise result here, from \cite{rw4}, is:

\begin{proposition}
Let $H_N^{\diamond k}\subset H_N^+$ be the easy quantum group coming from:
$$\pi_k=\ker\begin{pmatrix}1&\ldots&k&k&\ldots&1\\1&\ldots&k&k&\ldots&1\end{pmatrix}$$
Then $H_N^+=H_N^{\diamond 1}\supset H_N^{\diamond 2}\supset H_N^{\diamond 3}\supset\ldots\supset H_N^{[\infty]}$, and we obtain in this way all the intermediate easy quantum groups $H_N^{[\infty]}\subset G\subset H_N^+$, satisfying $G\neq H_N^{[\infty]}$.
\end{proposition}

It remains to discuss the easy quantum groups $H_N\subset G\subset H_N^{[\infty]}$, with the endpoints $G=H_N,H_N^{[\infty]}$ included. We follow here \cite{rw1}, \cite{rw2}, \cite{rw3}. First, we have:

\begin{definition}
A reflection group $\Gamma=<g_1,\ldots,g_N>$ is called uniform if each permutation $\sigma\in S_N$ produces a group automorphism, $g_i\to g_{\sigma(i)}$.
\end{definition}

Given a uniform reflection group $\mathbb Z_2^{*N}\to\Gamma\to\mathbb Z_2^N$, we can associate to it a family of subsets $D(k,l)\subset P(k,l)$, which form a category of partitions, as follows:
$$D(k,l)=\left\{\pi\in P(k,l)\Big|\ker(^i_j)\leq\pi\implies g_{i_1}\ldots g_{i_k}=g_{j_1}\ldots g_{j_l}\right\}$$

Observe that we have $P_{even}^{[\infty]}\subset D\subset P_{even}$, with the inclusions coming respectively from $\eta\in D$, and from $\Gamma\to\mathbb Z_2^N$. Conversely, given a category of partitions $P_{even}^{[\infty]}\subset D\subset P_{even}$, we can associate to it a uniform reflection group $\mathbb Z_2^{*N}\to\Gamma\to\mathbb Z_2^N$, as follows:
$$\Gamma=\left\langle g_1,\ldots g_N\Big|g_{i_1}\ldots g_{i_k}=g_{j_1}\ldots g_{j_l},\forall i,j,k,l,\ker(^i_j)\in D(k,l)\right\rangle$$

As explained in \cite{rw2}, the correspondences $\Gamma\to D$ and $D\to\Gamma$ are bijective, and inverse to each other, at $N=\infty$. We have in fact the following result, from \cite{rw1}, \cite{rw2}, \cite{rw3}:

\begin{proposition}
We have correspondences between:
\begin{enumerate}
\item Uniform reflection groups $\mathbb Z_2^{*\infty}\to\Gamma\to\mathbb Z_2^\infty$.

\item Categories of partitions $P_{even}^{[\infty]}\subset D\subset P_{even}$.

\item Easy quantum groups $G=(G_N)$, with $H_N^{[\infty]}\supset G_N\supset H_N$.
\end{enumerate}
\end{proposition}

As an illustration, if we denote by $\mathbb Z_2^{\circ N}$ the quotient of $\mathbb Z_2^{*N}$ by the relations of type $abc=cba$ between the generators, we have the following correspondences:
$$\xymatrix@R=15mm@C=15mm{
\mathbb Z_2^N\ar@{~}[d]&\mathbb Z_2^{\circ N}\ar[l]\ar@{~}[d]&\mathbb Z_2^{*N}\ar[l]\ar@{~}[d]\\
H_N\ar[r]&H_N^*\ar[r]&H_N^{[\infty]}}$$

More generally, for any $s\in\{2,4,\ldots,\infty\}$, the quantum groups $H_N^{(s)}\subset H_N^{[s]}$ constructed in \cite{bc1} come from the quotients of $\mathbb Z_2^{\circ N}\leftarrow\mathbb Z_2^{*N}$ by the relations $(ab)^s=1$. See \cite{rw3}.

We can now formulate a final classification result, as follows:

\begin{theorem}
The easy quantum groups $H_N\subset G\subset O_N^+$ are as follows,
$$\xymatrix@R=15mm@C=10mm{
O_N\ar[rr]&&O_N^*\ar[rr]&&O_N^+\\
H_N\ar[r]\ar[u]&H_N^\Gamma\ar[r]&H_N^{[\infty]}\ar[r]&H_N^{\diamond  k}\ar[r]&H_N^+\ar[u]}$$
with the family $H_N^\Gamma$ covering $H_N,H_N^{[\infty]}$, and with the series $H_N^{\diamond k}$ covering $H_N^+$.
\end{theorem}

This follows indeed from Proposition 1.5, Proposition 1.6 and Proposition 1.8 above. For further details, we refer to the paper of Raum and Weber \cite{rw4}.

\section{Twisting, intersections}

We recall from \cite{ba3} that the signature map $\varepsilon: P_{even}\to\{-1,1\}$, extending the usual signature of permutations, $\varepsilon:S_\infty\to\{-1,1\}$, is obtained by setting $\varepsilon(\pi)=(-1)^c$, where $c\in\mathbb N$ is the number of switches between neighbors required for making $\pi$ noncrossing, and which can be shown to be well-defined modulo 2. See \cite{ba3}.

We can make act permutations on tensors in a twisted way, as follows:

\begin{definition}
Associated to any partition $\pi\in P_{even}(k,l)$ is the linear map
$$\bar{T}_\pi(e_{i_1}\otimes\ldots\otimes e_{i_k})=\sum_{\tau\leq\pi}\varepsilon(\tau)\sum_{j:\ker(^i_j)=\tau}e_{j_1}\otimes\ldots\otimes e_{j_l}$$
where $\varepsilon:P_{even}\to\{-1,1\}$ is the signature map.
\end{definition}

Observe the similarity with Definition 1.1. In fact, the maps $T_\pi$ can be obtained as above, by stating that ``the untwisted signature is by definition 1, for all partitions''.

Here are a few basic examples of such maps, taken from \cite{ba3}:

\begin{proposition}
The linear maps associated to the basic crossings are:
$$\bar{T}_{\slash\!\!\!\backslash}(e_i\otimes e_j)
=\begin{cases}
-e_j\otimes e_i&{\rm for}\ i\neq j\\
e_j\otimes e_i&{\rm otherwise}
\end{cases}$$
$$\ \ \ \ \ \ \ \ \ \bar{T}_{\slash\hskip-1.6mm\backslash\hskip-1.1mm|\hskip0.5mm}(e_i\otimes e_j\otimes e_k)
=\begin{cases}
-e_k\otimes e_j\otimes e_i&{\rm for}\ i,j,k\ {\rm distinct}\\
e_k\otimes e_j\otimes e_i&{\rm otherwise}
\end{cases}$$
Also, for any noncrossing partition $\pi\in NC_{even}$ we have $\bar{T}_\pi=T_\pi$.
\end{proposition}

\begin{proof}
The basic crossings $\slash\!\!\!\backslash=\ker(^{ab}_{ba}),\slash\hskip-2.0mm\backslash\hskip-1.7mm|=\ker(^{abc}_{cba})$ are both odd, because they have respectively $1,3$ crossings, and their various subpartitions are as follows:
$$\ker\begin{pmatrix}a&a\\a&a\end{pmatrix},\ \ker\begin{pmatrix}a&a&b\\b&a&a\end{pmatrix},\ \ker\begin{pmatrix}a&b&a\\a&b&a\end{pmatrix},\ \ker\begin{pmatrix}b&a&a\\a&a&b\end{pmatrix},\ \ker\begin{pmatrix}a&a&a\\a&a&a\end{pmatrix}$$ 

Since all these subpartitions are even, we obtain the formulae in the statement. As for the second assertion, this comes from $\tau\leq\pi\in NC_{even}\implies\varepsilon(\tau)=1$. See \cite{ba3}.
\end{proof}

The idea now is that we can twist the easy quantum groups $H_N\subset G\subset O_N^+$, by using the linear maps in Definition 2.1. We should perhaps mention here that the twisting operation is usually dealt with by using cocycles, see e.g. \cite{bbc}. However, for our present purposes, we will rather need a ``Schur-Weyl twisting'', which is more powerful.

In order to define the twists, best to proceed as follows:

\begin{definition}
Associated to $H_N\subset G\subset O_N^+$ is its twist $H_N\subset\bar{G}\subset O_N^+$, given by
$$Hom(u^{\otimes k},u^{\otimes l})=span(\bar{T}_\pi|\pi\in D(k,l))$$
for any $k,l\in\mathbb N$, where $D\subset P$ is the category of partitions for $G$.
\end{definition}

Here we have used Woronowicz's Tannakian duality in \cite{wo2}. Indeed, as explained in \cite{ba3}, the correspondence $\pi\to\bar{T}_\pi$ is categorical, so the linear spaces in the statement form a tensor category, which produces via \cite{wo2} a compact quantum group $\bar{G}\subset O_N^+$. The fact that we have $H_N\subset\bar{G}$ comes from the equality $H_N=\bar{H}_N$, established in \cite{ba4}, and explained in Proposition 2.6 below, since by functoriality, $H_N=\bar{H}_N\subset\bar{G}$.

Here are some basic examples of such twists, coming from \cite{ba3}, \cite{bbc}: 

\begin{proposition}
$\bar{O}_N,\bar{O}_N^*\subset O_N^+$ are obtained respectively by imposing the relations
$$ab=\begin{cases}
-ba&{\rm for}\ a\neq b\ {\rm on\ the\ same\ row\ or\ column\ of\ }u\\
ba&{\rm otherwise}
\end{cases}$$
$$\hskip-20.9mm abc=\begin{cases}
-cba&{\rm for\ }r\leq2,s=3{\rm\ or\ }r=3,s\leq2\\
cba&{\rm for\ }r\leq2,s\leq 2{\rm\ or\ }r=s=3
\end{cases}$$
where $r,s\in\{1,2,3\}$ are the number of rows/columns of $u$ spanned by $a,b,c\in\{u_{ij}\}$.
\end{proposition}

\begin{proof}
Assuming that $G\subset O_N^+$ appears via the relations $T_\pi\in Hom(u^{\otimes k},u^{\otimes l})$, for a certain partition $\pi\in P(k,l)$, its twist $\bar{G}\subset O_N^+$ appears via the relations $\bar{T}_\pi\in Hom(u^{\otimes k},u^{\otimes l})$. Thus $\bar{O}_N,\bar{O}_N^*$ appear respectively via the relations $\bar{T}_{\slash\!\!\!\backslash}\in End(u^{\otimes 2})$, $\bar{T}_{\slash\hskip-1.6mm\backslash\hskip-1.1mm|\hskip0.5mm}\in End(u^{\otimes 3})$, and the result follows from the formulae in Proposition 2.2 above. See \cite{ba3}.
\end{proof}

We will show in what follows that $\bar{O}_N,\bar{O}_N^*$ are in fact the only possible twists. Let us first examine the case of $H_N,H_N^*,H_N^{[\infty]},H_N^+$, with some direct methods, based on signature computations that we will need as well later on, in section 4 below. We have:

\begin{lemma}
We have the following formulae
\begin{eqnarray*}
P_{even}^{[\infty]}&=&\left\{\pi\in P_{even}\Big|\varepsilon(\tau)=1,\forall\tau\leq\pi\right\}\\
P_{even}^*&=&\left\{\pi\in P_{even}\Big|\varepsilon(\tau)=1,\forall\tau\leq\pi,|\tau|=2\right\}
\end{eqnarray*}
where $|.|$ denotes the number of blocks.
\end{lemma}

\begin{proof}
We first prove the second equality. Given $\pi\in P_{even}$, we have $\tau\leq\pi,|\tau|=2$ precisely when $\tau=\pi^\beta$ is the partition obtained from $\pi$ by merging all the legs of a certain subpartition $\beta\subset\pi$, and by merging as well all the other blocks. Now observe that $\pi^\beta$ does not depend on $\pi$, but only on $\beta$, and that the number of switches required for making $\pi^\beta$ noncrossing is $c=N_\bullet-N_\circ$ modulo 2, where $N_\bullet/N_\circ$ is the number of black/white legs of $\beta$, when labelling the legs of $\pi$ counterclockwise $\circ\bullet\circ\bullet\ldots$ Thus $\varepsilon(\pi^\beta)=1$ holds precisely when $\beta\in\pi$ has the same number of black and white legs, and this gives the result.

We prove now the first equality. We recall from section 1 that we have:
$$P_{even}^{[\infty]}(k,l)=\left\{\ker\begin{pmatrix}i_1&\ldots&i_k\\ j_1&\ldots&j_l\end{pmatrix}\Big|g_{i_1}\ldots g_{i_k}=g_{j_1}\ldots g_{j_l}\ {\rm inside}\ \mathbb Z_2^{*N}\right\}$$

In other words, the partitions in $P_{even}^{[\infty]}$ are those describing the relations between free variables, subject to the conditions $g_i^2=1$. We conclude that $P_{even}^{[\infty]}$ appears from $NC_{even}$ by ``inflating blocks'', in the sense that each $\pi\in P_{even}^{[\infty]}$ can be transformed into a partition $\pi'\in NC_{even}$ by deleting pairs of consecutive legs, belonging to the same block. 

Now since this inflation operation leaves invariant modulo 2 the number $c\in\mathbb N$ of switches in the definition of the signature, it leaves invariant the signature $\varepsilon=(-1)^c$ itself, and we obtain in this way the inclusion ``$\subset$'' in the statement. 

Conversely, given $\pi\in P_{even}$ satisfying $\varepsilon(\tau)=1$, $\forall\tau\leq\pi$, our claim is that:
$$\rho\leq\sigma\subset\pi,|\rho|=2\implies\varepsilon(\rho)=1$$

Indeed, let us denote by $\alpha,\beta$ the two blocks of $\rho$, and by $\gamma$ the remaining blocks of $\pi$, merged altogether. We know that the partitions $\tau_1=(\alpha\wedge\gamma,\beta)$, $\tau_2=(\beta\wedge\gamma,\alpha)$, $\tau_3=(\alpha,\beta,\gamma)$ are all even. On the other hand, putting these partitions in noncrossing form requires respectively $s+t,s'+t,s+s'+t$ switches, where $t$ is the number of switches needed for putting $\rho=(\alpha,\beta)$ in noncrossing form. Thus $t$ is even, and we are done.

With the above claim in hand, we conclude, by using the second equality in the statement, that we have $\sigma\in P_{even}^*$. Thus we have $\pi\in P_{even}^{[\infty]}$, which ends the proof of ``$\supset$''.
\end{proof}

With the above lemma in hand, we can now prove:

\begin{proposition}
The basic quantum groups $H_N\subset G\subset H_N^+$, namely
$$H_N\subset H_N^*\subset H_N^{[\infty]}\subset H_N^+$$
are equal to their own twists.
\end{proposition}

\begin{proof}
We know from section 1 that the corresponding categories of partitions are:
$$P_{even}\supset P_{even}^*\supset P_{even}^{[\infty]}\supset NC_{even}$$

With this observation in hand, the proof goes as follows:

(1) $H_N^+$. We know from Proposition 2.2 for $\pi\in NC_{even}$ we have $\bar{T}_\pi=T_\pi$, and since we are in the situation $D\subset NC_{even}$, the definitions of $G,\bar{G}$ coincide.

(2) $H_N^{[\infty]}$. Here we can use the same argument as in (1), based this time on the description of $P_{even}^{[\infty]}$ found in Lemma 2.5 above.

(3) $H_N^*$. We have $H_N^*=H_N^{[\infty]}\cap O_N^*$, so $\bar{H}_N^*\subset H_N^{[\infty]}$ is the subgroup obtained via the defining relations for $\bar{O}_N^*$. But all the $abc=-cba$ relations defining $\bar{H}_N^*$ are automatic, of type $0=0$, and it follows that $\bar{H}_N^*\subset H_N^{[\infty]}$ is the subgroup obtained via the relations $abc=cba$, for any $a,b,c\in\{u_{ij}\}$. Thus we have $\bar{H}_N^*=H_N^{[\infty]}\cap O_N^*=H_N^*$, as claimed.

(4) $H_N$. We have $H_N=H_N^*\cap O_N$, and by functoriality, $\bar{H}_N=\bar{H}_N^*\cap\bar{O}_N=H_N^*\cap\bar{O}_N$. But this latter intersection was shown in \cite{ba4} to be equal to $H_N$, as claimed.
\end{proof}

In order to investigate now the general case, we need to establish the precise relation between the maps $T_\pi,\bar{T}_\pi$. By using the formulae in Proposition 2.2, we obtain:
$$\bar{T}_{\slash\!\!\!\backslash}=-T_{\slash\!\!\!\backslash}+2T_{\ker(^{aa}_{aa})}$$
$$\bar{T}_{\slash\hskip-1.6mm\backslash\hskip-1.1mm|\hskip0.5mm}=-\bar{T}_{\slash\hskip-1.6mm\backslash\hskip-1.1mm|\hskip0.5mm}+2T_{\ker(^{aab}_{baa})}+2T_{\ker(^{aba}_{aba})}+2T_{\ker(^{baa}_{aab})}-4T_{\ker(^{aaa}_{aaa})}$$

In general, the answer comes from the M\"obius inversion formula. We recall that the M\"obius function of any lattice, and in particular of $P_{even}$, is given by:
$$\mu(\sigma,\pi)=\begin{cases}
1&{\rm if}\ \sigma=\pi\\
-\sum_{\sigma\leq\tau<\pi}\mu(\sigma,\tau)&{\rm if}\ \sigma<\pi\\
0&{\rm if}\ \sigma\not\leq\pi
\end{cases}$$

With this notation, we have the following result:

\begin{lemma}
For any partition $\pi\in P_{even}$ we have the formula
$$\bar{T}_\pi=\sum_{\tau\leq\pi}\alpha_\tau T_\tau$$
where $\alpha_\sigma=\sum_{\sigma\leq\tau\leq\pi}\varepsilon(\tau)\mu(\sigma,\tau)$, with $\mu$ being the M\"obius function of $P_{even}$.
\end{lemma}

\begin{proof}
The linear combinations $T=\sum_{\tau\leq\pi}\alpha_\tau T_\tau$ acts on tensors as follows:
\begin{eqnarray*}
T(e_{i_1}\otimes\ldots\otimes e_{i_k})
&=&\sum_{\tau\leq\pi}\alpha_\tau T_\tau(e_{i_1}\otimes\ldots\otimes e_{i_k})\\
&=&\sum_{\tau\leq\pi}\alpha_\tau\sum_{\sigma\leq\tau}\sum_{j:\ker(^i_j)=\sigma}e_{j_1}\otimes\ldots\otimes e_{j_l}\\
&=&\sum_{\sigma\leq\pi}\left(\sum_{\sigma\leq\tau\leq\pi}\alpha_\tau\right)\sum_{j:\ker(^i_j)=\sigma}e_{j_1}\otimes\ldots\otimes e_{j_l}
\end{eqnarray*}

Thus, in order to have $\bar{T}_\pi=\sum_{\tau\leq\pi}\alpha_\tau T_\tau$, we must have, for any $\sigma\leq\pi$:
$$\varepsilon(\sigma)=\sum_{\sigma\leq\tau\leq\pi}\alpha_\tau$$

But this problem can be solved by using the M\"obius inversion formula, and we obtain the numbers $\alpha_\sigma=\sum_{\sigma\leq\tau\leq\pi}\varepsilon(\tau)\mu(\sigma,\tau)$ in the statement.
\end{proof}

Now back to the general twisting problem, the answer here is:

\begin{proposition}
The twists of the easy quantum groups $H_N\subset G\subset O_N^+$ are:
\begin{enumerate}
\item For $G=O_N,O_N^*$ we obtain $\bar{G}=\bar{O}_N,\bar{O}_N^*$.

\item For $G\neq O_N,O_N^*$ we have $G=\bar{G}$.
\end{enumerate}
\end{proposition}

\begin{proof}
We use the classification result in Theorem 1.9 above. We have to examine the 3 cases left, namely $G=O_N^+,H_N^{\diamond k},H_N^\Gamma$, and the proof goes as follows:

(1) Let $G=O_N^+$. We know from Proposition 2.2 for $\pi\in NC_{even}$ we have $\bar{T}_\pi=T_\pi$, and since we are in the situation $D\subset NC_{even}$, the definitions of $G,\bar{G}$ coincide.

(2) Let $G=H_N^{\diamond k}$. We know from Proposition 1.6 that the generating partition is:
$$\pi_k=\ker\begin{pmatrix}1&\ldots&k&k&\ldots&1\\1&\ldots&k&k&\ldots&1\end{pmatrix}$$

By symmetry, putting this partition in noncrossing form requires the same number of upper switches and lower switches, and so requires an even number of total switches. Thus $\pi_k$ is even, and the same argument shows in fact that all its subpartitions are even as well. It follows that we have $T_{\pi_k}=\bar{T}_{\pi_k}$, and this gives the result.

(3) Let $G=H_N^\Gamma$. We denote by $P_{even}^{[\infty]}\subset D\subset P_{even}$ the corresponding category of partitions. According to the description of $P_{even}^{[\infty]}$ worked out in \cite{bc1}, and mentioned after Definition 1.6 above, this category contains the following type of partition:
$$
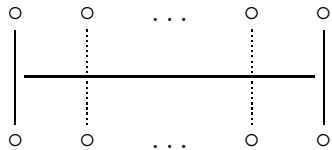
\xymatrix@R=5mm@C=5mm{
\circ\ar@{-}[dd]&\circ\ar@{.}[dd]&\ldots&\circ\ar@{.}[dd]&\circ\ar@{-}[dd]\\
\ar@{-}[rrrr]&&&&\\
\circ&\circ&\ldots&\circ&\circ}$$

The point now is that, by ``capping'' with such partitions, we can merge any pair of blocks of $\pi\in D$, by staying inside $D$. Thus, $D$ has the following property:
$$\tau\leq\pi\in D\implies\tau\in D$$

We deduce from this and from Lemma 2.7 that $\bar{T}_\pi$ is an intertwiner for $G$, and so $G\subset\bar{G}$. By symmetry we must have $\bar{G}\subset G$ as well, and this finishes the proof.
\end{proof}

As explained in \cite{ba3}, \cite{ba4}, the theory of ``easy noncommutative spheres'', first developed in \cite{bgo}, can be extended by twisting, and then by taking intersections between twisted and untwisted objects. We can proceed similarly with the quantum groups themselves:

\begin{theorem}
The easy quantum groups $H_N\subset G\subset O_N^+$ and their twists are
$$
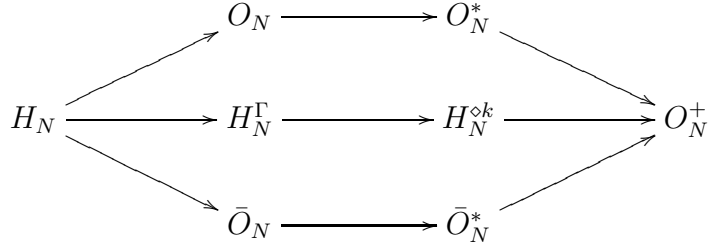
\xymatrix@R=7mm@C=20mm{
&O_N\ar[r]&O_N^*\ar[rd]\\
H_N\ar[r]\ar[ur]\ar[rd]&H_N^\Gamma\ar[r]&H_N^{\diamond  k}\ar[r]&O_N^+\\
&\bar{O}_N\ar[r]&\bar{O}_N^*\ar[ru]}$$
and the set formed by these quantum groups is stable by intersections.
\end{theorem}

\begin{proof}
According to Proposition 2.8 the easy quantum groups $H_N\subset G\subset O_N^+$ and their twists are the quantum groups in Theorem 1.9 and the twists $\bar{O}_N,\bar{O}_N^*$ from Proposition 2.4. But these are exactly the quantum groups in the above diagram.

Regarding now the intersection assertion, we can use here some computations from \cite{ba4}. We recall from there that we have the following intersection diagram:
$$
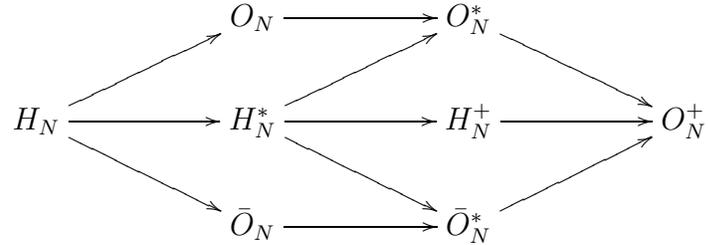
\xymatrix@R=7mm@C=20mm{
&O_N\ar[r]&O_N^*\ar[rd]\\
H_N\ar[r]\ar[ur]\ar[rd]&H_N^*\ar[r]\ar[ur]\ar[dr]&H_N^+\ar[r]&O_N^+\\
&\bar{O}_N\ar[r]&\bar{O}_N^*\ar[ru]}$$

More precisely, this diagram has the property that any intersection $G\cap H$ appears on the diagram, as the biggest quantum group contained in both $G,H$. See \cite{ba4}.

With this diagram in hand, the assertion follows. Indeed, the intersections between the quantum groups $O_N^\times$ are their twists are all on this diagram, and hence on the diagram in the statement as well. Regarding now the intersections of an easy quantum group $H_N\subset G\subset H_N^+$ with the twists $\bar{O}_N,\bar{O}_N^*$, we can use again the above diagram. Indeed, from $H_N^+\cap\bar{O}_N^*=H_N^*$ we deduce that both $K=G\cap\bar{O}_N,K'=G\cap\bar{O}_N^*$ appear as intermediate easy quantum groups $H_N\subset K^\times\subset H_N^*$, and we are done.
\end{proof}

\section{Noncommutative cubes}

In this section and in the next one we introduce our main objects of study, the noncommutative cubes and spheres. These are some special algebraic submanifolds of the free sphere $S^{N-1}_{\mathbb R,+}$, constructed in \cite{bgo}. We will first introduce $S^{N-1}_{\mathbb R,+}$ and a number of related spheres, from \cite{ba3}, \cite{ba4}, \cite{bgo}, and then we will discuss the noncommutative cubes. The noncommutative spheres will be further discussed in the next section.

Our starting point is the following definition, going back to \cite{bgo}:

\begin{definition}
The free real sphere $S^{N-1}_{\mathbb R,+}$ is defined by the following formula:
$$C(S^{N-1}_{\mathbb R,+})=C^*\left(x_1,\ldots,x_N\Big|x_i=x_i^*,x_1^2+\ldots+x_N^2=1\right)$$
Its half-liberated version $S^{N-1}_{\mathbb R,*}\subset S^{N-1}_{\mathbb R,+}$ is obtained by assuming $x_ix_jx_k=x_kx_jx_i$.
\end{definition}

Observe that we have inclusions $S^{N-1}_\mathbb R\subset S^{N-1}_{\mathbb R,*}\subset S^{N-1}_{\mathbb R,+}$. It is known from \cite{bgo}, \cite{bg1} that the corresponding quantum isometry groups are $O_N\subset O_N^*\subset O_N^+$. A twisted version of this result was established in \cite{ba3}. Further results include the construction of the eigenspaces of the Laplacian. We will be back later on to some of these topics, with full details.

Let us restrict now attention to the algebraic submanifolds $X\subset S^{N-1}_{\mathbb R,+}$. These are defined in analogy with the usual algebraic manifolds $X\subset S^{N-1}_\mathbb R$, as follows:

\begin{definition}
A closed subset $X\subset S^{N-1}_{\mathbb R,+}$ is called algebraic when
$$C(X)=C(S^{N-1}_{\mathbb R,+})/<P_1,P_2,\ldots>$$
where $P_i$ are noncommutative polynomials in the variables $x_1,\ldots,x_N$.
\end{definition}

As a first example, observe that the subspheres $S^{N-1}_\mathbb R,S^{N-1}_{\mathbb R,*}\subset S^{N-1}_{\mathbb R,+}$ are both algebraic, because they appear respectively from the following polynomials:
\begin{eqnarray*}
P_{ij}&=&x_ix_j-x_jx_i\\
P_{ijk}&=&x_ix_jx_k-x_kx_jx_i
\end{eqnarray*}

Observe also that, the usual sphere $S^{N-1}_\mathbb R$ being algebraic in the above sense, any algebraic submanifold $X\subset S^{N-1}_\mathbb R$ is as well algebraic in the above sense.

Another class of examples are the noncommutative cubes. Let us begin with:

\begin{proposition}
Any quotient $\mathbb Z_2^{*N}\to\Gamma\to\mathbb Z_2^N$ can be presented with relations
$$M_\alpha(g_1,\ldots,g_N)=N_\alpha(g_1,\ldots,g_N)$$
with the noncommutative monomials $M_\alpha,N_\alpha$ having the same degree in each variable.
\end{proposition}

\begin{proof}
Let $M_\alpha=N_\alpha$ be one of the relations presenting $\Gamma$, as a quotient of $\mathbb Z_2^{*N}$. This relation is by definition of the following type, for certain multi-indices $i,j$:
$$g_{i_1}\ldots g_{i_k}=g_{j_1}\ldots g_{j_l}$$

Since we have a quotient map $\Gamma\to\mathbb Z_2^N$ we deduce that we have $\ker(^i_j)\in P_{even}$, and by replacing where needed the variables $g_i$ by variables of type $g_i^s$ with $s$ odd, we can obtain a relation $M_\alpha'=N_\alpha'$ which is equivalent to $M_\alpha=N_\alpha$, as in the statement.
\end{proof}

We call a presentation as above ``normalized''. With this convention, we have:

\begin{proposition}
Given a reflection group $\mathbb Z_2^{*N}\to\Gamma\to\mathbb Z_2^N$, its dual is an algebraic manifold $\widehat{\Gamma}\subset S^{N-1}_{\mathbb R,+}$, with coordinates $x_i=\frac{g_i}{\sqrt{N}}$. Moreover, we have embeddings
$$\xymatrix@R=15mm@C=20mm
{S^{N-1}_\mathbb R\ar[rr]&&S^{N-1}_{\mathbb R,+}\\
\widehat{\mathbb Z_2^N}\ar[u]\ar[r]&\widehat{\Gamma}\ar[r]&\widehat{\mathbb Z_2^{*N}}\ar[u]}$$
with $\widehat{\Gamma}\subset\widehat{\mathbb Z_2^{*N}}$ appearing via the normalized group relations for $\Gamma$, and with $\widehat{\mathbb Z_2^N}\subset S^{N-1}_\mathbb R$ appearing as the standard cube/sphere inclusion, $\{x\in\mathbb R^N|x_i=\pm\frac{1}{\sqrt{N}},\forall i\}\subset S^{N-1}_\mathbb R$.
\end{proposition}

\begin{proof}
Since $\Gamma$ is a reflection group, we have $g_i=g_i^*$, $g_1^2=1$ inside the group algebra $C^*(\Gamma)=C(\widehat{\Gamma})$, and we deduce that $x_i=\frac{g_i}{\sqrt{N}}$ defines indeed an embedding $\widehat{\Gamma}\subset S^{N-1}_{\mathbb R,+}$.

Regarding now the diagram in the statement, we can construct it by using these maps $x_i=\frac{g_i}{\sqrt{N}}$, for the groups $\Gamma=\mathbb Z_2^N,\mathbb Z_2^{*N}$, at left and at right, and by dualizing the quotient maps $\mathbb Z_2^{*N}\to\Gamma\to\mathbb Z_2^N$ in order to construct the inclusions on the bottom.

The assertion about $\widehat{\Gamma}\subset\widehat{\mathbb Z_2^{*N}}$, which proves in particular that $\widehat{\Gamma}$ is an algebraic manifold, is clear as well. Indeed, the quotient map $C^*(\mathbb Z_2^{*N})\to C^*(\Gamma)$ comes by imposing the relations $M_\alpha=N_\alpha$ to the group elements $g_i$, and by assuming that these relations are normalized, this is the same as imposing them to the coordinates $x_i=\frac{g_i}{\sqrt{N}}$.

Finally, regarding the last assertion, for $\Gamma=\mathbb Z_2^N$ the space $\widehat{\Gamma}$ is classical, so by abelianizing, the embedding $\widehat{\Gamma}\subset S^{N-1}_{\mathbb R,+}$ must come from an embedding $\widehat{\Gamma}\subset S^{N-1}_\mathbb R$. Morever, since this latter embedding is given by $x_i=\frac{g_i}{\sqrt{N}}$, the points in its image must satisfy $x_i^2=\frac{1}{N}$ for any $i$, so the image is contained in $\{x\in\mathbb R^N|x_i=\pm\frac{1}{\sqrt{N}},\forall i\}$. Now since this latter set has the same cardinality as $\widehat{\Gamma}$, namely $2^N$, we obtain the result.
\end{proof}

We will be interested in computing the quantum isometry groups of the noncommutative cubes $\widehat{\Gamma}$, and of some related noncommutative spheres as well. We use here:

\begin{definition}
An affine action of an orthogonal quantum group $G\subset O_N^+$ on a closed subset $X\subset S^{N-1}_{\mathbb R,+}$ corresponds by definition to a coaction map 
$$\Phi:C(X)\to C(G)\otimes C(X)$$
given by $\Phi(x_i)=\sum_ju_{ij}\otimes x_j$, where $x_i,u_{ij}$ are the standard coordinates of $X,G$.
\end{definition}

In the classical case, it is well-known that any isometry of a closed subset $X\subset S^{N-1}_\mathbb R$ is affine. If we assume in addition that $X$ is  non-degenerate, in the sense that its coordinates $x_1,\ldots,x_N\in C(X)$ are linearly independent, then different affine isometries $U\in O_N$ of $X$ will have different restrictions $U_{|X}:X\to X$, and so the usual isometry group $G(X)$ is isomorphic to the biggest subgroup $G\subset O_N$ acting affinely on $X$. Moreover, as explained by Goswami in \cite{go2}, the quantum isometry group $G^+(X)$, taken in a metric space sense, exists, and is isomorphic to the biggest subgroup $G\subset O_N^+$ acting affinely on $X$.

In the general case, $X\subset S^{N-1}_{\mathbb R,+}$, no such results are  available, and this due to several technical difficulties, still waiting to be overcomed. See \cite{chi}, \cite{go2}, \cite{qsa}, \cite{rw2}. For the purposes of the present paper, best is to proceed as follows:

\begin{proposition}
Let $X\subset S^{N-1}_{\mathbb R,+}$ be algebraic, and non-degenerate, in the sense that the coordinates $x_1,\ldots,x_N\in C(X)$ are linearly independent. Then the quantum group
$$G^+(X)=\max\left\{G\subset O_N^+\Big|G\curvearrowright X\right\}$$
exists. We call it quantum (affine) isometry group of $X$.
\end{proposition}

\begin{proof}
The relations defining $G^+(X)$ being those making $x_i\to X_i=\sum_ju_{ij}\otimes x_j$ a morphism of algebras, we first have to clarify how the relations $P_i(X_1,\ldots,X_N)=0$ are interpreted inside $C(O_N^+)$. So, pick one of these polynomials, $P=P_i$, and write it:
$$P(x_1,\ldots,x_N)=\sum_r\alpha_r\cdot x_{i_1^r}\ldots x_{i_{s(r)}^r}$$

When replacing each $x_i\in C(X)$ by the element $X_i=\sum_ju_{ij}\otimes x_j\in C(O_N^+)\otimes C(X)$, we obtain the following formula:
$$P(X_1,\ldots,X_N)=\sum_r\alpha_r\sum_{j_1^r\ldots j_{s(r)}^r}u_{i_1^rj_1^r}\ldots u_{i_{s(r)}^rj_{s(r)}^r}\otimes x_{j_1^r}\ldots x_{j_{s(r)}^r}$$

If we set $k=\max_rs(r)$, then we have $P(X_1,\ldots,X_N)\in C(O_N^+)\otimes E_k$, where $E_k\subset C(X)$ is the linear space given by the following formula:
$$E_k=span\left(x_{i_1}\ldots x_{i_s}\Big|s\leq k\right)$$

Now since this space $E_k$ is finite dimensional, the relations $P(X_1,\ldots,X_N)=0$ correspond indeed to certain polynomial relations between the generators $u_{ij}$ of the algebra $C(O_N^+)$, and this finishes the proof of the existence/uniqueness of $G^+(X)$.

It remains to verify that the closed subspace $G^+(X)\subset O_N^+$ that we have constructed is indeed a quantum group. For this purpose, consider the following elements:
$$u_{ij}^\Delta=\sum_ku_{ik}\otimes u_{kj}\quad,\quad
u_{ij}^\varepsilon=\delta_{ij}\quad,\quad
u_{ij}^S=u_{ji}$$

Here, with $A=C(G^+(X))$, the elements $u_{ij}^\Delta$ belong by definition to $A\otimes A$, the elements $u_{ij}^\varepsilon$ belong to $\mathbb C$, and the elements $u_{ij}^S$ belong to the opposite algebra $A^{opp}$.

Now if we consider the associated elements $X_i^\gamma=\sum_ju_{ij}^\gamma\otimes x_j$, with $\gamma\in\{\Delta,\varepsilon,S\}$, then from $P(X_1,\ldots,X_N)=0$ we deduce that we have:
$$P(X_1^\gamma,\ldots,X_N^\gamma)=(\gamma\otimes id)P(X_1,\ldots,X_N)=0$$

Thus, by using the universal property of $G^+(X)$, we can construct morphisms of algebras mapping $u_{ij}\to u_{ij}^\gamma$ for any $\gamma\in\{\Delta,\varepsilon,S\}$, and this finishes the proof.
\end{proof} 

Let us first examine the basic examples of quantum groups $G^+(\widehat{\Gamma})$. The results here, some of them being already known from \cite{bbc}, \cite{bs2}, are as follows:

\begin{theorem}
The quantum isometry groups of basic noncommutative cubes are
$$\xymatrix@R=15mm@C=20mm{
\widehat{\mathbb Z_2^N}\ar[r]\ar@{~}[d]&\widehat{\mathbb Z_2^{\circ N}}\ar[r]\ar@{~}[d]&\widehat{\mathbb Z_2^{*N}}\ar@{~}[d]\\
\bar{O}_N\ar@{.}[r]&H_N^*\ar[r]&H_N^+}$$
with all arrows being inclusions, and with no map at bottom left.
\end{theorem}

\begin{proof}
The results in the classical and free cases are known from \cite{bbc}, \cite{bs2}, and the half-liberated result is new. We will present here complete proofs for all the results.

In all cases we must find the conditions on a closed subgroup $G\subset O_N^+$ such that $g_i\to \sum_ju_{ij}\otimes g_j$ defines a coaction. Since the coassociativity of such a map is automatic, we are left with checking that the map itself exists, and this is the same as checking that the variables $G_i=\sum_ju_{ij}\otimes g_j$ satisfy the same relations as the generators $g_i\in G$.

(1) For $\Gamma=\mathbb Z_2^N$ the relations to be checked are $G_i^2=1,G_iG_j=G_jG_i$. We have:
\begin{eqnarray*}
G_i^2&=&\sum_{kl}u_{ik}u_{il}\otimes g_kg_l=1+\sum_{k<l}(u_{ik}u_{il}+u_{il}u_{ik})\otimes g_kg_l\\
\left[G_i,G_j\right]&=&\sum_{k<l}(u_{ik}u_{jl}-u_{jk}u_{il}+u_{il}u_{jk}-u_{jl}u_{ik})\otimes g_kg_l
\end{eqnarray*}

From the first relation we obtain $ab=0$ for $a\neq b$ on the same row of $u$, and by using the antipode, the same happens for the columns. From the second relation we obtain $[u_{ik},u_{jl}]=[u_{jk},u_{il}]$ for $k\neq l$. Now by applying the antipode we obtain $[u_{lj},u_{ki}]=[u_{li},u_{kj}]$, and by relabelling, this gives $[u_{ik},u_{jl}]=[u_{il},u_{jk}]$ for $j\neq i$. Thus for $i\neq j,k\neq l$ we must have $[u_{ik},u_{jl}]=[u_{jk},u_{il}]=0$, and we are therefore led to $G\subset\bar{O}_N$, as claimed.

(2) For $\Gamma=\mathbb Z_2^{\circ N}$ the relations to be checked are $G_i^2=1,G_iG_jG_k=G_kG_jG_i$. With the notation $[a,b,c]=abc-cba$, we have:
\begin{eqnarray*}
G_i^2&=&\sum_{kl}u_{ik}u_{il}\otimes g_kg_l=1+\sum_{k\neq l}u_{ik}u_{il}\otimes g_kg_l\\
\left[G_i,G_j,G_k\right]&=&\sum_{abc}[u_{ia},u_{jb},u_{kc}]\otimes g_ag_bg_c
\end{eqnarray*}

From the first relation we obtain $G\subset H_N^+$. In order to process now the second relation, we can split the sum over $a,b,c$ in the following way:
\begin{eqnarray*}
\left[G_i,G_j,G_k\right]
&=&\sum_{a,b,c\ distinct}[u_{ia},u_{jb},u_{kc}]\otimes g_ag_bg_c+\sum_{a\neq b}[u_{ia},u_{jb},u_{ka}]\otimes g_ag_bg_a\\
&+&\sum_{a\neq c}[u_{ia},u_{ja},u_{kc}]\otimes g_c+\sum_{a\neq c}[u_{ia},u_{jc},u_{kc}]\otimes g_a\\
&+&\sum_a[u_{ia},u_{ja},u_{ka}]\otimes g_a
\end{eqnarray*}

Our claim is that the last three sums vanish. Indeed, $[u_{ia},u_{ja},u_{ka}]=\delta_{ijk}u_{ia}-\delta_{ijk}u_{ia}=0$, so the last sum vanishes. Regarding now the third sum, we have:
\begin{eqnarray*}
\sum_{a\neq c}[u_{ia},u_{ja},u_{kc}]
&=&\sum_{a\neq c}u_{ia}u_{ja}u_{kc}-u_{kc}u_{ja}u_{ia}
=\sum_{a\neq c}\delta_{ij}u_{ia}^2u_{kc}-\delta_{ij}u_{kc}u_{ia}^2\\
&=&\delta_{ij}\sum_{a\neq c}[u_{ia}^2,u_{kc}]=\delta_{ij}\left[\sum_{a\neq c}u_{ia}^2,u_{kc}\right]=\delta_{ij}[1-u_{ic}^2,u_{kc}]=0
\end{eqnarray*}

The proof for the fourth sum is similar. Thus, we are left with the first two sums. By using $g_ag_bg_c=g_cg_bg_a$ for the first sum, the formula becomes:
\begin{eqnarray*}
\left[G_i,G_j,G_k\right]
&=&\sum_{a<c,b\neq a,c}\left([u_{ia},u_{jb},u_{kc}]+[u_{ic},u_{jb},u_{ka}]\right)\otimes g_ag_bg_c\\
&+&\sum_{a\neq b}[u_{ia},u_{jb},u_{ka}]\otimes g_ag_bg_a
\end{eqnarray*}

In order to have a coaction, the above coefficients must vanish. Now observe that, when setting $a=c$ in the coefficients of the first sum, we obtain twice the coefficients of the second sum. Thus, our vanishing conditions can be formulated as follows:
$$[u_{ia},u_{jb},u_{kc}]+[u_{ic},u_{jb},u_{ka}]=0,\forall b\neq a,c$$

Now observe that at $i=j$ or $j=k$ this condition reads $0+0=0$. Thus, we can formulate our vanishing conditions in a more symmetric way, as follows:
$$[u_{ia},u_{jb},u_{kc}]+[u_{ic},u_{jb},u_{ka}]=0,\forall j\neq i,k,\forall b\neq a,c$$

We use now a trick from \cite{bg1}. We apply the antipode to this formula, and then we relabel the indices $i\leftrightarrow c,j\leftrightarrow b,k\leftrightarrow a$. We succesively obtain in this way:
$$[u_{ck},u_{bj},u_{ai}]+[u_{ak},u_{bj},u_{ci}]=0,\forall j\neq i,k,\forall b\neq a,c$$
$$[u_{ia},u_{jb},u_{kc}]+[u_{ka},u_{jb},u_{ic}]=0,\forall b\neq a,c,\forall j\neq i,k$$

Since we have $[a,b,c]=-[c,b,a]$, by comparing the last formula with the original one, we conclude that our vanishing relations reduce to a single formula, as follows:
$$[u_{ia},u_{jb},u_{kc}]=0,\forall j\neq i,k,\forall b\neq a,c$$

Our first claim is that this formula implies $G\subset H_N^{[\infty]}$. In order to prove this, we will just need the $c=a$ particular case of this formula, which reads:
$$u_{ia}u_{jb}u_{ka}=u_{ka}u_{jb}u_{ia},\forall j\neq i,k,\forall a\neq b$$

We know from \cite{bc1} that $H_N^{[\infty]}\subset O_N^+$ is defined via the relations $xyz=0$, for any $x\neq z$ on the same row or column of $u$. Thus, in order to prove that we have $G\subset H_N^{[\infty]}$, it is enough to check that the assumptions $j\neq i,k$ and $a\neq b$ can be dropped. But this is what happens indeed, because at $j=i$, $j=k$, $a=b$, we respectively have:
\begin{eqnarray*}
\left[u_{ia},u_{ib},u_{ka}\right]&=&u_{ia}u_{ib}u_{ka}-u_{ka}u_{ib}u_{ia}=\delta_{ab}(u_{ia}^2u_{ka}-u_{ka}u_{ia}^2)=0\\
\left[u_{ia},u_{kb},u_{ka}\right]&=&u_{ia}u_{kb}u_{ka}-u_{ka}u_{kb}u_{ia}=\delta_{ab}(u_{ia}u_{ka}^2-u_{ka}^2u_{ia})=0\\
\left[u_{ia},u_{ja},u_{ka}\right]&=&u_{ia}u_{ja}u_{ka}-u_{ka}u_{ja}u_{ia}=\delta_{ijk}(u_{ia}^3-u_{ia}^3)=0
\end{eqnarray*}

Our second claim now is that, due to $G\subset H_N^{[\infty]}$, we can drop the assumptions $j\neq i,k$ and $b\neq a,c$ in the original relations $[u_{ia},u_{jb},u_{kc}]=0$. Indeed, at $j=i$ we have:
$$[u_{ia},u_{ib},u_{kc}]=u_{ia}u_{ib}u_{kc}-u_{kc}u_{ib}u_{ia}=\delta_{ab}(u_{ia}^2u_{kc}-u_{kc}u_{ia}^2)=0$$

The proof at $j=k$ and at $b=a$, $b=c$ being similar, this finishes the proof of our claim. We conclude that the half-commutation relations $[u_{ia},u_{jb},u_{kc}]=0$ hold without any assumption on the indices, and so we obtain $G\subset H_N^*$, as claimed.

(3) For $\Gamma=\mathbb Z_2^{*N}$ the only relations to be checked are $G_i^2=1$. But these relations can be processed as in the proof of (2) above, and we obtain $G\subset H_N^+$, as claimed.
\end{proof}

The above computations, along with those in \cite{rw2}, lead to a number of interesting questions. We will be back to these questions in section 5 below, by jointly investigating them for the noncommutative cubes, and for the related noncommutative spheres.

\section{Noncommutative spheres}

In this section we upgrade the noncommutative sphere formalism from \cite{ba3}, \cite{ba4}, \cite{bgo}. The idea will be to replace the permutations $\sigma\in S_\infty\subset P_2$ used there by more general partitions $\pi\in P_{even}$. Our starting point is the following definition, from \cite{ba3}, \cite{ba4}:

\begin{definition}
Associated to any permutation $\sigma\in S_k$ are the sets of relations:
$$\mathcal R_\sigma=\left\{x_{i_1}\ldots x_{i_k}=x_{i_{\sigma(1)}}\ldots x_{i_{\sigma(k)}}\Big|\forall i_1,\ldots,i_k\right\}$$
$$\bar{\mathcal R}_\sigma=\left\{x_{i_1}\ldots x_{i_k}=\varepsilon(\sigma)x_{i_{\sigma(1)}}\ldots x_{i_{\sigma(k)}}\Big|\forall i_1,\ldots,i_k\right\}$$
We call these the untwisted/twisted relations associated to $\sigma$.
\end{definition}

Here the relations are between abstract variables $x_1,\ldots,x_N$, and we use the signature map $\varepsilon:P_{even}\to\{-1,1\}$ from \cite{ba3}, that we already met in section 2 above. 

As a basic example, for the standard crossing $\slash\!\!\!\backslash=(21)\in S_2$, we have:
$$\mathcal R_{\slash\!\!\!\backslash}=\{x_ix_j=x_jx_i|\forall i,j\}$$
$$\ \ \ \ \ \bar{\mathcal R}_{\slash\!\!\!\backslash}=\{x_ix_j=-x_jx_i|\forall i\neq j\}$$

Also, for the half-liberating permutation $\slash\hskip-2.0mm\backslash\hskip-1.7mm|=(321)\in S_3$, we have:
$$\mathcal R_{\slash\hskip-1.6mm\backslash\hskip-1.1mm|\hskip0.5mm}=\{x_ix_jx_k=x_kx_jx_i|\forall i,j,k\}$$
$$\hskip2.0mm\bar{\mathcal R}_{\slash\hskip-1.6mm\backslash\hskip-1.1mm|\hskip0.5mm}=\left[x_ix_jx_k=\begin{cases}
-x_kx_jx_i&\forall i,j,k\ {\rm distinct}\\
x_kx_jx_i&{\rm otherwise}
\end{cases}
\right]$$

These formulae follow indeed by using the signature computations from the proof of Proposition 2.2 above. For further details, and more examples, see \cite{ba3}, \cite{ba4}.

The point now is that by using the relations in Definition 4.1 above we can construct several types of families of noncommutative spheres, as follows:

\begin{definition}
We have the following spheres $X\subset S^{N-1}_{\mathbb R,+}$:
\begin{enumerate}
\item Linear spheres: $S^{N-1}_{\mathbb R,G}$ with $G\subset S_\infty$, defined via $\{\mathcal R_\sigma|\sigma\in G\}$.

\item Twisted linear spheres: $\bar{S}^{N-1}_{\mathbb R,H}$ with $H\subset S_\infty$, defined via $\{\bar{\mathcal R}_\sigma|\sigma\in H\}$.

\item Mixed linear spheres: $S^{N-1}_{\mathbb R,G,H}=S^{N-1}_{\mathbb R,G}\cap\bar{S}^{N-1}_{\mathbb R,H}$, with $G,H\subset S_\infty$.
\end{enumerate}
\end{definition}

Observe that the linear spheres cover the key examples $S^{N-1}_\mathbb R\subset S^{N-1}_{\mathbb R,*}\subset S^{N-1}_{\mathbb R,+}$ from \cite{bgo}. The twisted linear spheres cover the twists $\bar{S}^{N-1}_\mathbb R\subset\bar{S}^{N-1}_{\mathbb R,*}\subset S^{N-1}_{\mathbb R,+}$ constructed in \cite{ba3}, and the mixed linear sphere formalism covers these 5 examples, plus 4 more examples, which appear by intersecting $S^{N-1}_\mathbb R,S^{N-1}_{\mathbb R,*}$ with $\bar{S}^{N-1}_\mathbb R,\bar{S}^{N-1}_{\mathbb R,*}$, as follows:
$$\xymatrix@R=12mm@C=12mm{
S^{N-1}_\mathbb R\ar[r]&S^{N-1}_{\mathbb R,*}\ar[r]&S^{N-1}_{\mathbb R,+}\\
S^{N-1,1}_\mathbb R\ar[r]\ar[u]&S^{N-1,1}_{\mathbb R,*}\ar[r]\ar[u]&\bar{S}^{N-1}_{\mathbb R,*}\ar[u]\\
S^{N-1,0}_\mathbb R\ar[r]\ar[u]&\bar{S}^{N-1,1}_\mathbb R\ar[r]\ar[u]&\bar{S}^{N-1}_\mathbb R\ar[u]}$$

Here all 9 spheres, including the 4 examples at bottom left, which appear as intersections, are particular cases of the following construction from \cite{ba4}, with $d\in\{1,\ldots,N\}$:
$$C(S^{N-1,d-1}_{\mathbb R,\times})=C(S^{N-1}_{\mathbb R,\times})\Big/\Big\langle x_{i_0}\ldots x_{i_d}=0,\forall i_0,\ldots,i_d\ {\rm distinct}\Big\rangle$$

The mixed linear spheres can be studied by using the following concept, from \cite{ba4}:

\begin{proposition}
Let $S=S^{N-1}_{\mathbb R,G,H}$ be a mixed linear sphere, and consider the subsets $\widetilde{G},\widetilde{H}\subset S_\infty$ consisting of permutations $\sigma,\rho$ such that $\mathcal R_\sigma,\bar{\mathcal R}_\rho$ hold over $S$.
\begin{enumerate}
\item $S=S^{N-1}_{\mathbb R,\widetilde{G},\widetilde{H}}$, and $(\widetilde{G},\widetilde{H})$ is maximal with this property.

\item $\widetilde{G},\widetilde{H}$ are both subgroups of $S_\infty$, stable under concatenation.
\end{enumerate}
We call the writing $S=S^{N-1}_{\mathbb R,G,H}$ with $G,H$ maximal ``standard parametrization'' of $S$.
\end{proposition}

\begin{proof}
Here the first assertion is clear from definitions, and the second assertion follows by suitably manipulating the corresponding relations. See \cite{ba4}.
\end{proof}

Among the main results in \cite{ba4} was the fact that the standard parametrization of the 9 main spheres involves only 3 permutation groups, namely $\{1\}\subset S_\infty^*\subset S_\infty$:

\begin{proposition}
The standard parametrization of the $9$ main spheres is
$$\xymatrix@R=10mm@C=10mm{
S_\infty\ar@{.}[d]&S_\infty^*\ar@{.}[d]&\{1\}\ar@{.}[d]&G/H\\
S^{N-1}_\mathbb R\ar[r]&S^{N-1}_{\mathbb R,*}\ar[r]&S^{N-1}_{\mathbb R,+}&\{1\}\ar@{.}[l]\\
S^{N-1,1}_\mathbb R\ar[r]\ar[u]&S^{N-1,1}_{\mathbb R,*}\ar[r]\ar[u]&\bar{S}^{N-1}_{\mathbb R,*}\ar[u]&S_\infty^*\ar@{.}[l]\\
S^{N-1,0}_\mathbb R\ar[r]\ar[u]&\bar{S}^{N-1,1}_\mathbb R\ar[r]\ar[u]&\bar{S}^{N-1}_\mathbb R\ar[u]&S_\infty\ar@{.}[l]}$$
where  $S_\infty^*=S_\infty\cap P_{even}^*$.
\end{proposition}

\begin{proof}
We refer to \cite{ba4} for the proof of this result, and for more information about $S_\infty^*$, with the remark that we will improve this result in Theorem 4.11 below.
\end{proof}

As explained in \cite{ba4}, the above result, and a number of further considerations regarding the subgroups $G\subset S_\infty$, suggest that, conjecturally, the 3 main examples of linear spheres are the only ones, the 3 main examples of twisted linear spheres are the only ones, and the 9 main examples of mixed linear spheres are the only ones. See \cite{ba4}.

Our purpose now will be that of extending the linear sphere formalism, by using more general partitions $\pi\in P_{even}$ instead of permutations $\sigma\in S_\infty$. We use:

\begin{definition}
We denote by $P_{vert}\subset P_{even}$ the set of partitions having the property that each block has the same number of upper and lower legs. 
\end{definition}

Observe that we have $S_\infty\subset P_{vert}$, and in fact $S_\infty=P_{vert}\cap P_2$. Observe also that, when switching between consecutive neighbors, as required for the computation of the signature, the partitions $\pi\in P_{vert}$ can be put in a very simple form, as follows:
$$
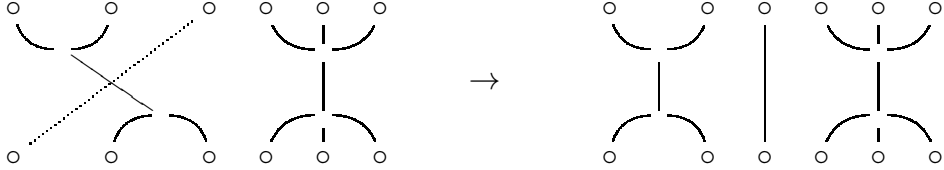
\xymatrix@R=2mm@C=3mm{
\circ\ar@/_/@{-}[dr]&&\circ&&\circ\ar@{.}[ddddllll]&\circ\ar@/_/@{-}[dr]&\circ\ar@{-}[d]&\circ\\
&\ar@/_/@{-}[ur]\ar@{-}[ddrr]&&&&&\ar@/_/@{-}[ur]\ar@{-}[dd]\\
\\
&&&\ar@/^/@{-}[dr]&&&\ar@/^/@{-}[dr]\ar@{-}[d]\\
\circ&&\circ\ar@/^/@{-}[ur]&&\circ&\circ\ar@/^/@{-}[ur]&\circ&\circ}
\qquad
\xymatrix@R=2mm@C=3mm{\\ \\ \to\\ \\}\ \ 
\qquad
\xymatrix@R=2mm@C=3mm{
\circ\ar@/_/@{-}[dr]&&\circ&\circ\ar@{-}[dddd]&\circ\ar@/_/@{-}[dr]&\circ\ar@{-}[d]&\circ\\
&\ar@/_/@{-}[ur]\ar@{-}[dd]&&&&\ar@/_/@{-}[ur]\ar@{-}[dd]\\
\\
&\ar@/^/@{-}[dr]&&&&\ar@/^/@{-}[dr]\ar@{-}[d]\\
\circ\ar@/^/@{-}[ur]&&\circ&\circ&\circ\ar@/^/@{-}[ur]&\circ&\circ}$$

We have in fact already met $P_{vert}$, in the proof of Proposition 3.3 above. Indeed, what we proved there is that any group $\mathbb Z_2^{*N}\to\Gamma\to\mathbb Z_2^N$ can be presented with relations of type $g_{i_1}\ldots g_{i_k}=g_{j_1}\ldots g_{j_l}$, with $\ker(^i_j)\in P_{vert}$. We will be back later on to this fact.

We can generalize the construction in Definition 4.1, as follows:

\begin{definition}
Associated to any $\pi\in P_{vert}$ are the sets of relations
$$\mathcal R_\pi=\left\{x_{i_1}\ldots x_{i_k}=x_{j_1}\ldots x_{j_k}\Big|\forall i,j,\ker(^i_j)\leq\pi\right\}$$
$$\bar{\mathcal R}_\pi=\left\{x_{i_1}\ldots x_{i_k}=\varepsilon\left(\ker(^i_j)\right)x_{j_1}\ldots x_{j_k}\Big|\forall i,j,\ker(^i_j)\leq\pi\right\}$$
which can be imposed to noncommutative variables $x_1,\ldots,x_N$.
\end{definition}

Observe that for $\pi\in S_\infty\subset P_{vert}$ we obtain indeed the relations in Definition 4.1 above. At the level of new examples, consider the following partitions:
$$\eta=\ker\begin{pmatrix}a&a&b\\ b&a&a\end{pmatrix}\quad,\quad
\nu=\ker\begin{pmatrix}a&b&a\\ a&b&a\end{pmatrix}\quad,\quad
\rho=\ker\begin{pmatrix}a&a&b\\ a&b&a\end{pmatrix}$$

Here $\eta$ is the pair-positioner partition, that we already met in Definition 1.4 above, and $\nu$ is a partition obtained by rotating it. These partitions are both even, and we have:
\begin{eqnarray*}
\mathcal R_\eta=\bar{\mathcal R}_\eta&=&\{x_i^2x_j=x_jx_i^2|\forall i,j\}\\
\mathcal R_\nu=\bar{\mathcal R}_\nu&=&\{x_ix_jx_i=x_ix_jx_i|\forall i,j\}
\end{eqnarray*}

Observe that, while $<\eta>=<\nu>$ by rotation, the above relations are of very different nature, with those for $\nu$ being trivial. This is in sharp contrast with the quantum group calculus developed in \cite{bsp}. Finally, for the above partition $\rho$, we have:
\begin{eqnarray*}
\mathcal R_\rho&=&\{x_i^2x_j=x_ix_jx_i|\forall i,j\}\\
\bar{\mathcal R}_\rho&=&\{x_i^2x_j=-x_ix_jx_i|\forall i\neq j\}
\end{eqnarray*}

Now back to the general case, with Definition 4.6 in hand, we can generalize in a straightforward way the constructions in Definition 4.2, as follows: 

\begin{definition}
We have the following spheres $X\subset S^{N-1}_{\mathbb R,+}$:
\begin{enumerate}
\item Monomial spheres: $S^{N-1}_{\mathbb R,E}$ with $E\subset P_{vert}$, defined via $\{\mathcal R_\pi|\pi\in E\}$.

\item Twisted monomial spheres: $\bar{S}^{N-1}_{\mathbb R,F}$ with $F\subset P_{vert}$, defined via $\{\bar{\mathcal R}_\pi|\pi\in F\}$.

\item Mixed monomial spheres: $S^{N-1}_{\mathbb R,E,F}=S^{N-1}_{\mathbb R,E}\cap\bar{S}^{N-1}_{\mathbb R,F}$, with $E,F\subset P_{vert}$.
\end{enumerate}
\end{definition}

At the classification level, we recall from \cite{ba4} that, conjecturally, the 3 main examples of linear spheres are the only ones, the 3 main examples of twisted linear spheres are the only ones, and the 9 main examples of mixed linear spheres are the only ones. In the monomial setting the situation is much more complicated, and we have no conjectural answer yet. We have for instance a big class of examples, constructed as follows:

\begin{definition}
Given a category of partitions $NC_{even}\subset C\subset P_{even}$, we construct the set $E_C=C\cap P_{vert}$, and then we associate:
\begin{enumerate}
\item To any $C$: the monomial sphere $S^{N-1}_C=S^{N-1}_{\mathbb R,E_C}$.

\item To any $D$: the twisted monomial sphere $\bar{S}^{N-1}_D=\bar{S}^{N-1}_{\mathbb R,E_D}$.

\item To any $C,D$: the mixed monomial sphere $S^{N-1}_{C,D}=S^{N-1}_{\mathbb R,E_C,E_D}$.
\end{enumerate}
\end{definition}

Observe the similarity with the concept of standard parametrization, from Proposition 4.3 above. Our purpose in what follows will be to clarify this similarity. 

Before doing so, however, let us discuss the main new example of monomial sphere appearing via Definition 4.8. This new sphere comes from $P_{vert}^{[\infty]}$, as follows:

\begin{proposition}
The monomial sphere $S^{N-1}_{\mathbb R,\infty}$ associated to $P_{vert}^{[\infty]}$ appears as: 
$$C(S^{N-1}_{\mathbb R,\infty})=C(S^{N-1}_{\mathbb R,+})/<x_i^2x_j=x_jx_i^2,\forall i,j>$$
Moreover, this sphere contains the half-liberated sphere $S^{N-1}_{\mathbb R,*}$.
\end{proposition} 

\begin{proof}
Observe first that the pair-positioner partition $\eta\in P_{vert}^{[\infty]}$ produces the relations $[a^2,b]=0$. In order to prove that $S^{N-1}_{\mathbb R,\infty}$ is indeed presented by these relations, we will need a convenient description of $P_{vert}^{[\infty]}$. We recall from section 1 that we have:
$$P_{even}^{[\infty]}(k,l)=\left\{\ker\begin{pmatrix}i_1&\ldots&i_k\\ j_1&\ldots&j_l\end{pmatrix}\Big|g_{i_1}\ldots g_{i_k}=g_{j_1}\ldots g_{j_l}\ {\rm inside}\ \mathbb Z_2^{*N}\right\}$$

In other words, the partitions in $P_{even}^{[\infty]}$ implement the relations $g_i^2=1$, between free variables $g_1,\ldots,g_N$. It follows that the partitions in $P_{vert}^{[\infty]}$ implement the relations $x_i^2=$ central, between free variables $x_1,\ldots,x_N$, and this gives the result.
\end{proof}

Now back to the parametrization question, observe that $P_{vert}\subset P_{even}$ is closed under the standard categorical operations $\circ,\otimes,*$ from \cite{bsp}, which are respectively the vertical and horizontal concatenation, and the upside-down turning. We can formulate:

\begin{proposition}
Let $S=S^{N-1}_{\mathbb R,E,F}$ be a mixed monomial sphere, and consider the subsets $\widetilde{E},\widetilde{F}\subset P_{vert}$ consisting of partitions $\pi,\sigma$ such that $\mathcal R_\pi,\bar{\mathcal R}_\sigma$ hold over $S$.
\begin{enumerate}
\item $S=S^{N-1}_{\mathbb R,\widetilde{E},\widetilde{F}}$, and $(\widetilde{E},\widetilde{F})$ is maximal with this property.

\item $\widetilde{E},\widetilde{F}\subset P_{vert}$ are both closed under the categorical operations $\circ,\otimes,*$.
\end{enumerate}
We call the writing $S=S^{N-1}_{\mathbb R,E,F}$ with $E,F$ maximal ``standard parametrization'' of $S$.
\end{proposition}

\begin{proof}
Here the first assertion is clear, and the second assertion follows as in the proof of Proposition 4.3, explained in detail in \cite{ba4}, by composing, concatenating, or returning the corresponding relations. Indeed, these operations correspond to the categorical operations $\circ,\otimes,*$, and so both $\widetilde{E},\widetilde{F}$ follow to be closed under these latter operations.
\end{proof}

We agree from now on to call the concepts used in Proposition 4.3 and Proposition 4.4 above ``old standard parametrization''.  Let us extend now Proposition 4.4, by using our new notion of standard parametrization, and by replacing as well the free sphere $S^{N-1}_{\mathbb R,+}$ with the smaller sphere $S^{N-1}_{\mathbb R,\infty}$. We have here the following result:

\begin{theorem}
We have the following standard parametrization results
$$\xymatrix@R=10mm@C=10mm{
P_{vert}\ar@{.}[d]&P_{vert}^*\ar@{.}[d]&P_{vert}^{[\infty]}\ar@{.}[d]&E/F\\
S^{N-1}_\mathbb R\ar[r]&S^{N-1}_{\mathbb R,*}\ar[r]&S^{N-1}_{\mathbb R,\infty}&P_{vert}^{[\infty]}\ar@{.}[l]\\
S^{N-1,1}_\mathbb R\ar[r]\ar[u]&S^{N-1,1}_{\mathbb R,*}\ar[r]\ar[u]&\bar{S}^{N-1}_{\mathbb R,*}\ar[u]&P_{vert}^*\ar@{.}[l]\\
S^{N-1,0}_\mathbb R\ar[r]\ar[u]&\bar{S}^{N-1,1}_\mathbb R\ar[r]\ar[u]&\bar{S}^{N-1}_\mathbb R\ar[u]&P_{vert}\ar@{.}[l]}$$
where $P_{vert}^*=P_{vert}\cap P_{even}^*$ and $P_{vert}^{[\infty]}=P_{vert}\cap P_{even}^{[\infty]}$.
\end{theorem}

\begin{proof}
The idea will be that of exploiting as much as possible Proposition 4.4, and then enhancing some of the arguments in the proof of Proposition 4.4, worked out in \cite{ba4}.

(I) First, we must prove that we have $S=S^{N-1}_{\mathbb R,E,F}$, for all the spheres in the statement. We will do this in two steps, first by converting the parametrization in Proposition 4.4 into a partition-theoretical statement, and then replacing $S^{N-1}_{\mathbb R,+}\to S^{N-1}_{\mathbb R,\infty}$.

In order to perform the first step, the idea is that of replacing in Proposition 4.4 the groups $S_\infty^\times=S_\infty,S_\infty^*,\{1\}$ by the sets $P_{vert}^\times=P_{vert},P_{vert}^*,NC_{vert}$. Our first claim is that these groups and sets are related by the following formulae:
\begin{eqnarray*}
S_\infty^\times&=&P_{vert}^\times\cap S_\infty\\
P_{vert}^\times&=&\left\{\pi\in P_{vert}\Big|\pi\leq\rho\ {\rm for\ some\ }\rho\in S_\infty^\times\right\}
\end{eqnarray*}

This is indeed clear in the classical case, clear as well in the free case, and in the half-liberated case this follows from the observation that, when labelling the legs counterclockwise $\circ\bullet\circ\bullet\ldots$, merging blocks will preserve the equality of black/white legs.

Now with these connecting formulae in hand, we deduce that we have the following equality, for any of the $3\times 3=9$ choices of the symbols $\times,\diamond$:
$$S^{N-1}_{\mathbb R,S_\infty^\times,S_\infty^\diamond}=S^{N-1}_{\mathbb R,P_{vert}^\times,P_{vert}^\diamond}$$

Indeed, the inclusion ``$\subset$'' comes from the first connecting formula, and the inclusion ``$\supset$'' comes from the second connecting formula. But, from this equality we conclude that the parametrization result in Proposition 4.4 can be reformulated as follows, with of course the parametrizing sets $E,F$ not claimed to be maximal:
$$\xymatrix@R=10mm@C=10mm{
P_{vert}\ar@{.}[d]&P_{vert}^*\ar@{.}[d]&NC_{vert}\ar@{.}[d]&E/F\\
S^{N-1}_\mathbb R\ar[r]&S^{N-1}_{\mathbb R,*}\ar[r]&S^{N-1}_{\mathbb R,+}&NC_{vert}\ar@{.}[l]\\
S^{N-1,1}_\mathbb R\ar[r]\ar[u]&S^{N-1,1}_{\mathbb R,*}\ar[r]\ar[u]&\bar{S}^{N-1}_{\mathbb R,*}\ar[u]&P_{vert}^*\ar@{.}[l]\\
S^{N-1,0}_\mathbb R\ar[r]\ar[u]&\bar{S}^{N-1,1}_\mathbb R\ar[r]\ar[u]&\bar{S}^{N-1}_\mathbb R\ar[u]&P_{vert}\ar@{.}[l]}$$

Let us insert now into this diagram the sphere left, $S^{N-1}_{\mathbb R,\infty}$. We know by definition that we have $S^{N-1}_{\mathbb R,\infty}=S^{N-1}_{P_{vert}^{[\infty]}}$, and our first claim is that we have in fact:
$$S^{N-1}_{\mathbb R,\infty}=S^{N-1}_{P_{vert}^{[\infty]},P_{vert}^{[\infty]}}$$

In order to prove this formula, it is enough to show that the relations $\bar{\mathcal R}_\pi$ in Definition 4.6 are satisfied over $S^{N-1}_{\mathbb R,\infty}$, for any $\pi\in P_{vert}^{[\infty]}$. But, according to the description $P_{even}^{[\infty]}$ found in Lemma 2.5 above, we obtain, by intersecting with $P_{vert}$:
$$P_{vert}^{[\infty]}=\left\{\pi\in P_{vert}\Big|\varepsilon(\tau)=1,\forall\tau\leq\pi\right\}$$

Now since the difference between the relations $\mathcal R_\pi,\bar{\mathcal R}_\pi$ in Definition 4.6 comes precisely from the possible odd subpartitions $\tau\leq\pi$, we conclude that for $\pi\in P_{vert}^{[\infty]}$ we have $\mathcal R_\pi=\bar{\mathcal R}_\pi$, and so these relations are indeed satisfied over $S^{N-1}_{\mathbb R,\infty}$.

Thus our claim is proved, and $S^{N-1}_{\mathbb R,\infty}$ can be therefore inserted into the above diagram, at the place of $S^{N-1}_{\mathbb R,+}$, with parametrizing sets $E/F=P_{vert}^{[\infty]}/P_{vert}^{[\infty]}$. Now since the standard parametrization operation in Proposition 4.10 is functorial, we can change as well $NC_{vert}\to P_{vert}^{[\infty]}$ for the parametrizing sets of the $2+2$ ``smaller'' spheres, sitting below or at left of $S^{N-1}_{\mathbb R,\infty}$, and we obtain in this way the diagram in the statement.

(II) We must prove now that the parametrization $S=S^{N-1}_{\mathbb R,E,F}$ in the statement is standard, for all the 9 spheres. We already know from Proposition 4.4 above that the intersections $G=E\cap S_\infty,H=F\cap S_\infty$ are the correct ones, and in order to extend this result, best is to fine-tune the proof of Proposition 4.4, done in detail in \cite{ba4}. We must compute the following sets, and show that we get the sets in the statement:
\begin{eqnarray*}
E&=&\left\{\pi\in P_{vert}\Big|{\rm the\ relations\ }\mathcal R_\pi\ {\rm hold\ over\ }S\right\}\\
F&=&\left\{\pi\in P_{vert}\Big|{\rm the\ relations\ }\bar{\mathcal R}_\pi\ {\rm hold\ over\ }S\right\}
\end{eqnarray*} 

As a first observation, by using the various inclusions between spheres, we just have to compute $E$ for the spheres on the bottom, and $F$ for the spheres on the left:
\begin{eqnarray*}
S=S^{N-1,0}_\mathbb R,\bar{S}^{N-1,1}_\mathbb R,\bar{S}^{N-1}_\mathbb R&\implies&E=P_{vert},P_{vert}^*,P_{vert}^{[\infty]}\\
S=S^{N-1,0}_\mathbb R,S^{N-1,1}_\mathbb R,S^{N-1}_\mathbb R&\implies&F=P_{vert},P_{vert}^*,P_{vert}^{[\infty]}
\end{eqnarray*}

The results for $S^{N-1,0}_\mathbb R$ being clear, we are left with computing the remaining 4 sets, for the spheres $S^{N-1}_\mathbb R,\bar{S}^{N-1}_\mathbb R,S^{N-1,1}_\mathbb R,\bar{S}^{N-1,1}_\mathbb R$. The proof here goes as follows:

(1) $S^{N-1}_\mathbb R$. According to the definition of $F$, we have:
\begin{eqnarray*}
F(k)
&=&\left\{\pi\in P_{vert}(k)\Big|x_{i_1}\ldots x_{i_k}=\varepsilon\left(\ker(^i_j)\right)x_{j_1}\ldots x_{j_k},\forall\ker(^i_j)\leq\pi\right\}\\
&=&\left\{\pi\in P_{vert}(k)\Big|\varepsilon\left(\ker(^i_j)\right)=1,\forall\ker(^i_j)\leq\pi\right\}
\end{eqnarray*}

Now since by Lemma 2.5 for any $\pi\in P_{vert}(k)-P_{vert}^{[\infty]}(k)$ we can find a partition $\tau\leq\pi$ satisfying $\varepsilon(\tau)=-1$, we deduce that we have $F=P_{vert}^{[\infty]}$, as desired.

(2) $\bar{S}^{N-1}_\mathbb R$. The proof of $E=P_{vert}^{[\infty]}$ here is similar to the proof of $F=P_{vert}^{[\infty]}$ in (1) above, by using the same combinatorial ingredient at the end.

(3) $S^{N-1,1}_\mathbb R$. By definition of $F$, a partition $\pi\in P_{vert}(k)$ belongs to $F(k)$ when the following condition is satisfied, for any choice of the indices satisfying $\ker(^i_j)\leq\pi$:
$$x_{i_1}\ldots x_{i_k}=\varepsilon\left(\ker(^i_j)\right)x_{j_1}\ldots x_{j_k}$$

When $|\ker i|=1$ this formula reads $x_r^k=x_r^k$, which is true. When $|\ker i|\geq3$ this formula is automatically satisfied as well, because by using the relations $ab=ba$, and $abc=0$ for $a,b,c$ distinct, which both hold over $S^{N-1,1}_\mathbb R$, this formula reduces to $0=0$. Thus, we are left with studying the case $|\ker i|=2$. Here the quantities on the left $x_{i_1}\ldots x_{i_k}$ will not vanish, so the sign on the right must be 1, and we therefore have:
$$F(k)=\left\{\pi\in P_{vert}(k)\Big|\varepsilon(\tau)=1,\forall\tau\leq\pi,|\tau|=2\right\}$$

By using now Lemma 2.5 we conclude that we have $F=P_{vert}^*$, as desired.

(4) $\bar{S}^{N-1,1}_\mathbb R$. The proof of $E=P_{vert}^*$ here is similar to the proof of $F=P_{vert}^*$ in (3) above, by using the same combinatorial ingredient at the end.
\end{proof}

As an application of Theorem 4.11, let us go back to Definition 4.8, and try to find out what the main examples of such spheres are, in the untwisted case. In view of the bijection between easy quantum groups and categories of partitions, we can take as data here the basic quantum groups in Proposition 1.3 and Definition 1.4, namely:
$$\xymatrix@R=15mm@C=15mm{
O_N\ar[r]&O_N^*\ar[rr]&&O_N^+\\
H_N\ar[r]\ar[u]&H_N^*\ar[r]\ar[u]&H_N^{[\infty]}\ar[r]&H_N^+\ar[u]}$$

We can compute the associated spheres by using Theorem 4.11, and we get:

\begin{proposition}
The spheres associated to the basic quantum groups are:
$$\xymatrix@R=15mm@C=15mm{
S^{N-1}_\mathbb R\ar[r]&S^{N-1}_{\mathbb R,*}\ar[rr]&&S^{N-1}_{\mathbb R,+}\\
S^{N-1}_\mathbb R\ar[r]\ar@{=}[u]&S^{N-1}_{\mathbb R,*}\ar[r]\ar@{=}[u]&S^{N-1}_{\mathbb R,\infty}\ar[r]&S^{N-1}_{\mathbb R,+}\ar@{=}[u]}$$
In particular, for any $H_N\subset G\subset O_N^+$ easy we have $S^{N-1}_G=S^{N-1}_{G'}$, with $G'=G\cap H_N^+$.
\end{proposition}

\begin{proof}
Observe first that the second assertion is clear from the first one, and from the classification result in Theorem 1.9, with the remark that in this second assertion we have used the notation $S^{N-1}_C$ from Definition 4.8 (1) with the category $NC_{even}\subset C\subset P_{even}$ replaced by the corresponding easy quantum group $H_N\subset G\subset O_N^+$.

Regarding now the first assertion, the computation here goes as follows:
$$\xymatrix@R=7mm@C=10mm{
H_N\ar@{.}[d]\ar[r]&H_N^*\ar[r]\ar@{.}[d]&H_N^{[\infty]}\ar[r]\ar@{.}[d]&H_N^+\ar@{.}[d]\\
P_{even}\ar@{.}[d]&P_{even}^*\ar[l]\ar@{.}[d]&P_{even}^{[\infty]}\ar[l]\ar@{.}[d]&NC_{even}\ar[l]\ar@{.}[d]\\
P_{vert}\ar@{.}[d]&P_{vert}^*\ar[l]\ar@{.}[d]&P_{vert}^{[\infty]}\ar[l]\ar@{.}[d]&NC_{vert}\ar[l]\ar@{.}[d]\\
S^{N-1}_\mathbb R\ar[r]&S^{N-1}_{\mathbb R,*}\ar[r]&S^{N-1}_{\mathbb R,\infty}\ar[r]&S^{N-1}_{\mathbb R,+}}\qquad\quad
\xymatrix@R=7.8mm@C=10mm{
O_N\ar@{.}[d]\ar[r]&O_N^*\ar[r]\ar@{.}[d]&O_N^+\ar@{.}[d]\\
P_2\ar@{.}[d]&P_2^*\ar[l]\ar@{.}[d]&NC_2\ar[l]\ar@{.}[d]\\
S_\infty\ar@{.}[d]&S_\infty^*\ar[l]\ar@{.}[d]&\{1\}\ar[l]\ar@{.}[d]\\
S^{N-1}_\mathbb R\ar[r]&S^{N-1}_{\mathbb R,*}\ar[r]&S^{N-1}_{\mathbb R,+}}$$

More precisely, the rows in these diagrams describe the corresponding categories of partitions $C$, the intersections $E_G=C\cap P_{vert}$, and finally the associated spheres.

The passage from the first row to the second row is clear from definitions, and so is the passage from the second row to the third row. As for the passage from the third row to the fourth row, this comes from Theorem 4.11, and finishes the proof. 
\end{proof}

Summarizing, we have extended the formalism in \cite{ba3}, \cite{ba4}, \cite{bgo}, and in the untwisted case our main examples are the spheres associated to the easy quantum groups $H_N\subset G\subset H_N^+$, classified in Theorem 1.9. We will gradually study these spheres in section 5 below, by beginning with those associated to the quantum groups $H_N\subset H_N^\Gamma\subset H_N^+$.

\section{The duality principle}

We discuss now the duality principle, between the noncommutative cubes and spheres associated to the uniform reflection groups $\mathbb Z_2^{*N}\to\Gamma\to\mathbb Z_2^N$. The construction of the correspondence is as follows, by using the usual notations $g_i,x_i$ for the generators of $\Gamma$, respectively for the standard coordinates on the free sphere $S^{N-1}_{\mathbb R,+}$:

\begin{proposition}
We have a duality $\widehat{\Gamma}\leftrightarrow S^{N-1}_\Gamma$ between noncommutative cubes and spheres associated to the uniform reflection groups $\mathbb Z_2^{*N}\to\Gamma\to\mathbb Z_2^N$, given by
$$C(S^{N-1}_\Gamma)=C(S^{N-1}_{\mathbb R,+})\Big/\left<\begin{matrix}\ \ g_{i_1}\ldots g_{i_k}=g_{j_1}\ldots g_{j_k},\forall i,j\\ \implies x_{i_1}\ldots x_{i_k}=x_{j_1}\ldots x_{j_k},\forall i,j\end{matrix}\right>$$
and by $\widehat{\Gamma}=S^{N-1}_\Gamma\cap\widehat{\mathbb Z_2^{*N}}$. We have as well a twisted correspondence $\widehat{\Gamma}\leftrightarrow\bar{S}^{N-1}_\Gamma$, obtained similarly, by using instead the relations $x_{i_1}\ldots x_{i_k}=\varepsilon(\ker(^i_j))x_{j_1}\ldots x_{j_k}$.
\end{proposition}

\begin{proof}
We recall from sections 3-4 that we have an inclusion $\widehat{\Gamma}\subset S^{N-1}_\Gamma$, coming from a quotient map as follows, where $M_\alpha=N_\alpha$ are uniform relations presenting $\Gamma$:
$$\xymatrix@R=10mm@C=5mm{
C(S^{N-1}_\Gamma)\ar@{=}[r]\ar[d]&C(S^{N-1}_{\mathbb R,+})\Big/\left<M_\alpha(x_i)=N_\alpha(x_i)\right>\ \ \ \ \ \ \ \ \ \ \,\ar[d]\\
C(\widehat{\Gamma})\ar@{=}[r]&C(S^{N-1}_{\mathbb R,+})\Big/\left<x_i^2=\frac{1}{N},M_\alpha(x_i)=N_\alpha(x_i)\right>}$$

But this shows that the maps in the statement are inverse to each other, and hence proves the duality result. The proof of the twisted statement is similar.
\end{proof}

We discuss now the computation and comparison of the associated quantum isometry groups. We restrict attention to the untwisted case, the results in the twisted case being similar. For the spheres coming from the main 3 reflection groups, we have:

\begin{proposition}
The untwisted monomial spheres coming from the basic reflection groups, $\mathbb Z_2^N\leftarrow\mathbb Z_2^{\circ N}\leftarrow\mathbb Z_2^{*N}$, and the associated quantum isometry groups, are
$$\xymatrix@R=15mm@C=20mm{
S^{N-1}_\mathbb R\ar[r]\ar@{~}[d]&S^{N-1}_{\mathbb R,*}\ar[r]\ar@{~}[d]&S^{N-1}_{\mathbb R,\infty}\ar@{~}[d]\\
O_N\ar[r]&O_N^*\ar@{.}[r]&H_N^{[\infty]}}$$
with all arrows being inclusions, and with no map at bottom right.
\end{proposition}

\begin{proof}
According to Proposition 4.12, the spheres are those in the statement. Moreover, according to the results in \cite{bgo}, the two quantum groups at bottom left are the correct ones. Thus, we are left with proving that we have $G^+(S^{N-1}_{\mathbb R,\infty})=H_N^{[\infty]}$.

Let us set as usual $X_i=\sum_au_{ia}\otimes x_a$, with $C(G)=<u_{ia}>$. By doing some index manipulations as in the proof of Theorem 3.7, we obtain the following formula:
\begin{eqnarray*}
X_iX_iX_j
&=&\sum_{b\neq a,c}u_{ia}u_{ib}u_{jc}\otimes x_ax_bx_c\\
&+&\sum_{a\neq c}(u_{ic}u_{ic}u_{ja}+u_{ia}u_{ic}u_{jc})\otimes x_ax_cx_c\\
&+&\sum_au_{ia}u_{ia}u_{ja}\otimes x_ax_ax_a
\end{eqnarray*}

Thus, the equalities $X_iX_iX_j=X_jX_iX_i$ correspond to the following relations:
\begin{eqnarray*}
u_{ia}u_{ib}u_{jc}&=&u_{ja}u_{ib}u_{ic},\forall b\neq a,c\\
u_{ic}^2u_{ja}+u_{ia}u_{ic}u_{jc}&=&u_{jc}u_{ic}u_{ia}+u_{ja}u_{ic}^2,\forall a\neq c\\
u_{ia}^2u_{ja}&=&u_{ia}u_{ja}^2
\end{eqnarray*}

As a first remark, these relations are satisfied indeed for $H_N^{[\infty]}$. Our claim now, which will finish the proof, is that the middle relation by itself implies $G\subset H_N^{[\infty]}$. Consider indeed this middle relation, which is best written as follows:
$$[u_{ia},u_{ic},u_{jc}]=[u_{ja},u_{ic}^2],\forall i\neq j,\forall a\neq c$$

Observe that we have added the condition $i\neq j$, because at $i=j$ the formula is trivial. Now by applying the antipode, and then by relabelling $i\leftrightarrow c,j\leftrightarrow a$, we obtain: 
\begin{eqnarray*}
\left[u_{cj},u_{ci},u_{ai}\right]&=&[u_{ci}^2,u_{aj}],\forall i\neq j,\forall a\neq c\\
\left[u_{ia},u_{ic},u_{jc}\right]&=&[u_{ic}^2,u_{ja}],\forall c\neq a,\forall j\neq i
\end{eqnarray*}

Since we have $[a,b]=-[b,a]$, we conclude that the following must hold:
$$[u_{ia},u_{ic},u_{jc}]=[u_{ja},u_{ic}^2]=0,\forall i\neq j,\forall a\neq c$$

We will need only the second formula, namely $[u_{ja},u_{ic}^2]=0$ for $i\neq j$, $a\neq c$. Our claim is that the assumptions $i\neq j$, $a\neq c$ can be dropped. Indeed, by summing over $i\neq j$ we obtain $[u_{ja},1-u_{jc}^2]=0$, and so $[u_{ja},u_{jc}^2]=0$, and so the assumption $i\neq j$ can be dropped. Similarly, the assumption $a\neq c$ can be dropped as well. 

We conclude that $[u_{ja},u_{ic}^2]=0$ holds without any restrictions on the indices, and since these relations are those definining $H_N^{[\infty]}\subset O_N^+$, this finishes the proof.
\end{proof}

Observe the similarity of the above statement with Theorem 3.7, and notably the lack of functoriality. Let us first put these results together:

\begin{theorem}
The basic noncommutative cubes, and the associated spheres are
$$\xymatrix@R=15mm@C=15mm{
S^{N-1}_\mathbb R\ar[r]&S^{N-1}_{\mathbb R,*}\ar[r]&S^{N-1}_{\mathbb R,\infty}\\
\widehat{\mathbb Z_2^N}\ar[r]\ar[u]&\widehat{\mathbb Z_2^{\circ N}}\ar[r]\ar[u]&\widehat{\mathbb Z_2^{*N}}\ar[u]}
\ \ \xymatrix@R=9mm@C=5mm{\\ \ar@{~}[r]&\\&\\}\ \ 
\xymatrix@R=15mm@C=15mm{
O_N\ar@{.}[d]\ar[r]&O_N^*\ar@{.}[r]&H_N^{[\infty]}\ar[d]\\
\bar{O}_N&H_N^*\ar@{.}[l]\ar[r]\ar[u]&H_N^+}$$
with the diagram at right describing the corresponding quantum isometry groups.
\end{theorem}

\begin{proof}
This follows indeed by putting together Theorem 3.7 and Proposition 5.2.
\end{proof}

The problem is that the diagram on the right suffers from a severe lack of functoriality. Putting this diagram into a reasonable duality framework looks like a challenging problem, that we will discuss now. One idea for overcoming the difficulties, coming from \cite{rw2}, and also from \cite{ba4}, is that of using the following version of Proposition 3.6:

\begin{proposition}
Let $X\subset S^{N-1}_{\mathbb R,+}$ be algebraic, and non-degenerate, in the sense that the coordinates $x_1,\ldots,x_N\in C(X)$ are linearly independent. Then the quantum group
$$H^+(X)=\max\left\{G\subset H_N^{[\infty]}\Big|G\curvearrowright X\right\}$$
exists. We call it quantum reflection group of $X$.
\end{proposition}

\begin{proof}
Both the existence and the uniqueness statement are clear, because we can simply use here Proposition 3.6, and set $H^+(X)=G^+(X)\cap H_N^{[\infty]}$.
\end{proof}

The point with the above notion basically comes from the fact that, when replacing $G^+(X)\to H^+(X)$, the statement of Theorem 5.3 drastically simplifies:

\begin{proposition}
The basic noncommutative cubes, and the associated spheres are
$$\xymatrix@R=15mm@C=15mm{
S^{N-1}_\mathbb R\ar[r]&S^{N-1}_{\mathbb R,*}\ar[r]&S^{N-1}_{\mathbb R,\infty}\\
\widehat{\mathbb Z_2^N}\ar[r]\ar[u]&\widehat{\mathbb Z_2^{\circ N}}\ar[r]\ar[u]&\widehat{\mathbb Z_2^{*N}}\ar[u]}
\ \ \xymatrix@R=9mm@C=5mm{\\ \ar@{~}[r]&\\&\\}\ \ 
\xymatrix@R=15mm@C=15mm{
H_N\ar@{=}[d]\ar[r]&H_N^*\ar[r]&H_N^{[\infty]}\ar@{=}[d]\\
H_N\ar[r]&H_N^*\ar[r]\ar@{=}[u]&H_N^{[\infty]}}$$
with the diagram at right describing the corresponding quantum reflection groups.
\end{proposition}

\begin{proof}
This follows from Theorem 5.3 above, by intersecting with $H_N^{[\infty]}$, with some help at computing intersections coming from Theorem 2.9 and its proof.
\end{proof}

This result, and the computations in \cite{rw2}, suggest that we should have the following formula, valid for any uniform refection group $\mathbb Z_2^{*N}\to\Gamma\to\mathbb Z_2^N$:
$$H^+(S^{N-1}_\Gamma)=H^+(\widehat{\Gamma})=H_N^\Gamma$$

In addition, we believe that the ``basic'' groups used in Theorem 5.3 are in fact ``exceptional''. More precisely, for $\Gamma\neq\mathbb Z_2^{*N},\mathbb Z_2^{\circ N},\mathbb Z_2^N$, our conjecture is that we have:
$$G^+(S^{N-1}_\Gamma)=G^+(\widehat{\Gamma})=H_N^\Gamma$$

Regarding the quantum group $H_N^{\diamond k}$, our conjecture is that this appears as quantum isometry group of the sphere $S^{N-1}_{\mathbb R,(k)}\subset S^{N-1}_{\mathbb R,+}$ obtained via the following relations:
$$[a_1\ldots a_{k-2}b^2a_{k-2}\ldots a_1,c^2]=0$$

If all these conjectures are true, the quantum groups in Theorem 2.9 above would appear as quantum isometry (or reflection) groups of the following spheres:
$$\xymatrix@R=7mm@C=18mm{
&S^{N-1}_\mathbb R\ar[r]&S^{N-1}_{\mathbb R,*}\ar[rd]\\
S^{N-1,1}_\mathbb R\ar[r]\ar[ur]\ar[rd]&S_\Gamma\ar[r]&S^{N-1}_{\mathbb R,(k)}\ar[r]&S^{N-1}_{\mathbb R,+}\\
&\bar{S}^{N-1}_\mathbb R\ar[r]&\bar{S}^{N-1}_{\mathbb R,*}\ar[ru]}$$

We do not know on how to solve these questions. The main issues come from: (1) our lack of global standard parametrization results, for the spheres in Definition 4.8 above, (2)  our poor understanding of the full quantum isometry groups, in the reflection group case, and (3) our poor understanding of the spheres $S^{N-1}_{\mathbb R,(k)}$ introduced above. 

\section{Unitary extensions}

We discuss here some unitary extensions of the above results. As explained in \cite{ba3}, \cite{ba4}, the unitary case is considerably more complex than the real one, but some basic results can be obtained by ``mirroring'' the real ones. We will adopt the same strategy here. More precisely, we will briefly describe the results which can be obtained in this way, and leave the computations, details, and an overall complete study, for later on.

The starting point is the following definition, coming from \cite{ba3}:

\begin{definition}
The free complex sphere, and its free complex cube, are given by
$$\xymatrix@R=10mm@C=5mm{
C(S^{N-1}_{\mathbb C,+})\ar@{=}[r]\ar[d]&C^*\left(z_1,\ldots,z_N\Big|\sum_iz_iz_i^*=\sum_iz_i^*z_i=1\right)\ar[d]\\
C^*(F_N)\ar@{=}[r]&C^*\left(g_1,\ldots,g_N\Big|g_ig_i^*=g_i^*g_i=1,\forall i\right)\ \ \ \ \ \ }$$
with the vertical quotient map being given by $z_i=\frac{g_i}{\sqrt{N}}$.
\end{definition}

With these notions in hand, the idea is that the various results in sections 3-5 above extend to the complex case, the general principle being, as in \cite{ba3}, \cite{ba4}, that of replacing the real coordinates $x_i\in C(S^{N-1}_{\mathbb R,+})$ by the variables $x_i=z_i,z_i^*\in C(S^{N-1}_{\mathbb C,+})$.

As an example, the half-liberation $S^{N-1}_{\mathbb C,**}$ that we must use is the ``minimal'' one, obtained by imposing the relations $abc=cba$, for any $a,b,c\in\{z_i,z_i^*\}$. See \cite{ba3}.

Let us discuss the analogue of Theorem 5.3. The spheres $S^{N-1}_\mathbb C\subset S^{N-1}_{\mathbb C,**}\subset S^{N-1}_{\mathbb C,+}$ appear from the groups $\mathbb Z^N\leftarrow\mathbb Z^{\circ N}\leftarrow F_N$, via a categorical construction which is similar to the one in the real case, with $\mathbb Z^{\circ N}$ being obtained from $F_N$ by stating that the generators $g_1,\ldots,g_N$ and their inverses satisfy the relations $abc=cba$. We have then:

\begin{theorem}
The basic noncommutative cubes, and the associated spheres are
$$\xymatrix@R=15mm@C=15mm{
S^{N-1}_\mathbb C\ar[r]&S^{N-1}_{\mathbb C,**}\ar[r]&S^{N-1}_{\mathbb C,\infty}\\
\mathbb T^N\ar[r]\ar[u]&\widehat{\mathbb Z^{\circ N}}\ar[r]\ar[u]&\widehat{F_N}\ar[u]}
\ \ \xymatrix@R=9mm@C=5mm{\\ \ar@{~}[r]&\\&\\}\ \ 
\xymatrix@R=15mm@C=15mm{
U_N\ar@{.}[d]\ar[r]&U_N^{**}\ar@{.}[r]&K_N^{[\infty]}\ar[d]\\
\bar{U}_N&K_N^{**}\ar@{.}[l]\ar[r]\ar[u]&K_N^+}$$
with the diagram at right describing the corresponding quantum isometry groups.
\end{theorem}

\begin{proof}
This statement, which is similar to Theorem 5.3, can be deduced in a similar way. First of all, the underlying unitary easy quantum groups are as follows:
$$\xymatrix@R=15mm@C=15mm{
\mathbb Z^N\ar@{~}[d]&\mathbb Z^{\circ N}\ar[l]\ar@{~}[d]&F_N\ar[l]\ar@{~}[d]\\
K_N\ar[r]&K_N^{**}\ar[r]&K_N^{[\infty]}}$$

Here $K_N=\mathbb T\wr S_N$ is the complex analogue of $H_N=\mathbb Z_2\wr S_N$, and $K_N^{**},K_N^{[\infty]}$ are the corresponding analogues of $H_N^*,H_N^{[\infty]}$, introduced and studied in \cite{ba4}.

With this result in hand, the associated spheres are those in the statement. Regarding now the quantum isometry groups, the 2 results at top left are from \cite{ba3}, the result on top right can be deduced by suitably modifying the proof of Proposition 5.2, and the 3 results on the bottom can be obtained by adpating the computations in \cite{bs1}, \cite{bs2}.
\end{proof}

The various comments made after Theorem 5.3 above apply as well to the complex case. We have for instance an analogue of Proposition 5.5 above, obtained by restricting the attention to the ``reflection'' quantum groups $K^+(X)=G^+(X)\cap K_N^{[\infty]}$.

Observe that the above results put the computations in \cite{bs1}, \cite{bs2} under a new light. Indeed, for $\widehat{F_N}$ itself, the quantum isometry group computed here, which is $K_N^+$, is much simpler than the quantum group $H_{N,0}^+$ computed in \cite{bs1}. This is due to the fact that our notion of quantum isometry group here is taken in an affine complex sense.

\end{document}